\documentclass{article}
\usepackage{amsmath,amsthm,amssymb,amscd}
\usepackage{color}

\theoremstyle{definition}
\newtheorem{Thm}{Theorem}[section]
\newtheorem{Def}[Thm]{Definition}
\newtheorem{Lem}[Thm]{Lemma}
\newtheorem{Cor}[Thm]{Corollary}
\newtheorem{Ex}[Thm]{Example}
\newtheorem{Prop}[Thm]{Proposition}
\newtheorem{Rem}[Thm]{Remark}

\newtheorem{Not}[Thm]{Notation}

\def\Ker{\mathrm{Ker}}

\def\H{\mathbb{H}}
\def\K{\mathbb{K}}

\def\N{\mathbb{N}}

\def\area{\mathrm{Area^{rel}}}

\title{Notes on relatively hyperbolic groups and relatively quasiconvex subgroups}
\author{
Yoshifumi Matsuda\footnote
{Graduate School of Mathematical Sciences, 
University of Tokyo, 
3-8-1 Komaba, 
Meguro-ku, 
Tokyo, 
153-8914 Japan,
ymatsuda@ms.u-tokyo.ac.jp} \footnote{
The author is supported by the Global COE Program at Graduate School of Mathematical Sciences, the University of Tokyo, and Grant-in-Aid for Scientific Researches for Young Scientists (B) (No. 22740034), Japan Society of Promotion of Science.}, 
Shin-ichi Oguni\footnote
{Department of Mathematics, Faculty of Science, Ehime University,
2-5 Bunkyo-cho, 
Matsuyama, 
Ehime, 
790-8577 Japan,
oguni@math.sci.ehime-u.ac.jp} \footnote{
The author is supported by Grant-in-Aid for Scientific Researches for Young Scientists (B) 
(No. 24740045), Japan Society of Promotion of Science.}, 
Saeko Yamagata\footnote
{Faculty of Education and Human Sciences, Yokohama National University,
240-8501 Yokohama, Japan,
yamagata@ynu.ac.jp}
}
\date{}

\begin{document}

\maketitle
%%%%%%%%%%%%%%%%%%%%%%%%%%%%%%%%%%%%%%%%%%%%%%%%%
\begin{abstract}
We define relatively quasiconvex subgroups 
of relatively hyperbolic groups in the sense of Osin 
and show that such subgroups have expected properties. 
Also we state several definitions equivalent to 
the definition of relatively hyperbolic groups 
in the sense of Osin.\\


\noindent
2010MSC:
20F67;
20F65
\end{abstract}

\section{Introduction}
The notion of a relatively hyperbolic group 
was introduced in \cite{Gro87} and has been studied by many authors
(see for example \cite{Bow12} and \cite{Far98}).
In this note we adopt a definition of a relatively hyperbolic group by Osin 
(\cite[Definition 2.35]{Osi06a}): 
a group $G$ is hyperbolic relative to a family of subgroups $\H$ if 
$G$ is finitely presented relative to $\H$ and the relative Dehn function is linear.
From now on, unless otherwise stated, 
relatively hyperbolic groups mean those in the sense of Osin. 
The definition is equivalent to other several definitions of a relatively hyperbolic group
if $G$ is countable and $\H$ is a finite family (see for example \cite[Section 3]{Hru10}). 
Also a relatively quasiconvex subgroup of a relatively hyperbolic group is defined by several 
equivalent ways if $G$ is countable and $\H$ is a finite family 
(see for example \cite[Section 6]{Hru10}). 
The main aim of this note is to give an appropriate 
definition of a relatively quasiconvex subgroup of 
a relatively hyperbolic group, which allows the case where $G$ is uncountable
and $\H$ is an infinite family. 
Also we give several equivalent definitions of relatively hyperbolic groups. 

Throughout this note, 
$G$ is a group, $\H$ is a family of non-trivial subgroups of $G$
and $L$ is a subgroup of $G$ 
($\H$ may be empty and two elements of $\H$ may coincide as a subgroup of $G$).

Now we define relatively quasiconvex subgroups of relatively hyperbolic groups. 
The following technical definition is essential. 
\begin{Def}\label{**}
We say that $L$ satisfies {\it Condition $(b)$} with respect to $\H$ in $G$ 
if for any $y,y'\in G$ such that $Ly\neq Ly'$, $\H_{L,y,y'}$ is a finite subset of $\H$. 
Here 
\[
\H_{L,y,y'}:=\{H\in\H \ |\  L\cap yHy'^{-1}\neq \emptyset\}.
\]
\end{Def}

\begin{Def}\label{rqc}
Suppose that 
$G$ is hyperbolic relative to $\H$.
The subgroup $L$ is said to be {\it pre-quasiconvex relative to} $\H$ in $G$ with respect to 
a finite relative generating system $X$ of $(G,\H)$
if there exists a finite subset $Y\subset G$ satisfying the following:
for any mutually different $l_1, l_2\in L$ and 
any geodesic $c$ in the relative Cayley graph $\overline{\Gamma}(G,\H,X)$ from $l_1$ to $l_2$, 
all vertices on $c$ are contained in $L\cup LY$.
Here $LY:=\{g\in G \ |\ g=ly, l\in L, y\in Y\}$.

The subgroup $L$ is said to be {\it pre-quasiconvex relative to} $\H$ in $G$ if 
$L$ is pre-quasiconvex relative to $\H$ in $G$ with respect to 
some finite relative generating system $X$.

The subgroup $L$ is said to be {\it quasiconvex relative to} $\H$ in $G$ if
$L$ is pre-quasiconvex relative to $\H$ in $G$ and 
satisfies Condition $(b)$ with respect to $\H$ in $G$.

The subgroup $L$ is said to be {\it strongly quasiconvex relative to} $\H$ in $G$ if
$L$ is quasiconvex relative to $\H$ in $G$
and for any $g\in G$, there exists a finite subset $\H'$ of $\H$ such that 
any $H_1\in \H$ and any $H_2\in \H\setminus \H'$, 
$L\cap gH_1g^{-1}$ is finite and $L\cap gH_2g^{-1}$ is trivial.
\end{Def}
\noindent 
We remark that $L$ is (pre-)quasiconvex (resp. strongly quasiconvex) relative to $\H$ in $G$ 
with respect to every finite relative generating system 
when $L$ is (pre-)quasiconvex (resp. strongly quasiconvex) relative to $\H$ in $G$ 
(see Corollary \ref{choice'''}).

We note that $L$ satisfies Condition $(b)$ with respect to $\H$ in $G$ 
whenever $\H$ is finite. 
Osin (\cite[Definition 4.9]{Osi06a}) originally introduced 
relative quasiconvexity by using a word metric 
when $G$ is finitely generated. We remark that $\H$ consists of finitely many elements
when $G$ is finitely generated (\cite[Corollary 2.48]{Osi06a}). 
Hruska (\cite[Definition 6.10]{Hru10}) extended the notion of relative quasiconvexity 
to the case where $\H$ is finite and $G$ is countable 
by using a left invariant proper metric on $G$. 
Our notion of relative quasiconvexity extends theirs to the case where 
$G$ and $\H$ are general. 

The following are main theorems.
\begin{Thm}\label{cap}
Let $K$ be a subgroup of $G$. 
Suppose that $G$ is hyperbolic relative to $\H$. 
If both $K$ and $L$ are quasiconvex relative to $\H$ in $G$, 
then $K\cap L$ is also quasiconvex relative to $\H$ in $G$. 
\end{Thm}
\noindent 
This was known for the case where $\H$ is finite and $G$ is countable
(see \cite[Theorem 1.1]{Osi06a} and \cite[Theorem 1.2 (2)]{Hru10}).
  
\begin{Thm}\label{qc-undist}
Suppose that $G$ is hyperbolic relative to $\H$ in $G$.
Then we have the following. 
\begin{enumerate}
\item The subgroup $L$ is quasiconvex relative to $\H$ in $G$ if and only if  
$L$ is undistorted relative to $\H$.
\item The subgroup $L$ is strongly quasiconvex relative to $\H$ in $G$ if and only if  
$L$ is strongly undistorted relative to $\H$ in $G$.
In particular if $L$ is strongly quasiconvex relative to $\H$ in $G$, 
then $L$ is hyperbolic. 
\end{enumerate} 
\end{Thm}
\noindent
See Definition \ref{rund} for (strongly) relatively undistorted subgroups. 
Theorem \ref{qc-undist} (ii) was known if $G$ is finitely generated 
(see \cite[Theorems 1.9 and 1.10]{Osi06a}).
Also the part of `only if' in (i) is known if $\H$ is finite and $G$ is countable
(see \cite[Theorem 10.1]{Hru10} and also \cite[Lemma 2.22]{Yan12}). 

\begin{Thm}\label{4}
Suppose that $G$ is hyperbolic relative to $\H$ in $G$
and $L$ is quasiconvex relative to $\H$ in $G$. 
Then $L$ is naturally relatively hyperbolic. 
\end{Thm}
\noindent
See Theorem \ref{4'} for a more precise statement.
This was known for the case where $\H$ is finite and $G$ is countable
(see \cite[Theorem 1.2 (1)]{Hru10}, \cite[Theorem 1.8]{MP-W11a} and \cite[Lemma 2.27]{Yan12}).

\vspace{1cm}
In Section \ref{pre}, some terminologies 
related to relatively hyperbolic groups are recalled. 
Also we introduce Condition $(a)$ for $(G,\H)$ 
in order to state equivalent definitions of relatively hyperbolic groups.
In Section \ref{geometry}, 
we give several equivalent definitions of relatively hyperbolic groups 
(Proposition \ref{relhypeq}).  
In Section \ref{relqcdef}, 
relatively undistorted subgroups and relatively quasiconvex subgroups
are studied. In particular, 
Theorems \ref{cap}, \ref{qc-undist} and \ref{4} are shown.  

%%%%%%%%%%%%%%%%%%%%%%%%%%%%%%%%%%%%
\section{Preliminaries}\label{pre}

We recall definitions and properties of 
relative generating systems, relative presentations, 
relative Cayley graphs and coned-off Cayley graphs
in Sections 2.1 and 2.2. 
We mainly refer to \cite{Osi06a} (see also \cite{Far98}). 
In Section 2.3, we introduce Condition $(a)$ in order to
state equivalent definitions of relatively hyperbolic groups
in Proposition \ref{relhypeq}.

%%%%%%%%%%%%%%%%%%%%%%%%%%%%%%%%%%%%
\subsection{Relative generating systems and relative presentations}

For any set $X$, $F(X)$ denotes a free group generated by $X$. 
When we take a map $\psi_0 \colon X\to G$, 
we regard the set $X$ as a family of elements of $G$ by $\psi_0$
if there is no confusion. 
Note that two distinct elements of $X$ can be the same element of $G$. 
We call $\psi_0 \colon X\to G$ or $X$ a {\it system} of $G$.
A system $\psi_0 \colon X\to G$ is said to be {\it symmetric} if $X$ is endowed with an involution 
$X\ni x\mapsto x^{-1}\in X$ such that $\psi_0(x^{-1})=\psi_0(x)^{-1}$ for every $x\in X$.
A system $\psi_0$ induces the natural homomorphism $\psi_1 \colon F(X) \to G$.  
If $\psi_1 \colon F(X) \to G$ is surjective, then $\psi_0 \colon X\to G$ or $X$ is called 
a {\it generating system} of $G$. 
 
We denote by $\ast_{H\in\H} H$ the free product of all elements of $\H$. 
The inclusion from $H$ to $G$ for every $H\in \H$ 
induces the natural homomorphism $\psi_2 \colon \ast_{H\in \H}H \to G$. 
Let us denote the free product of $F(X)$ for a system $X$ and $\ast_{H\in\H}H$ 
by $F(X) \ast \left(\ast_{H\in\H}H\right)$ and set $F=F(X) \ast \left(\ast_{H\in\H}H\right)$. 
We have the natural homomorphism $\psi\colon F\to G$ induced by $\psi_1$ and $\psi_2$. 
If $\psi$ is surjective, 
then $\psi_0 \colon X \to G$ or $X$ is called 
a {\it relative generating system} of $(G,\H)$. 
If $(G,\H)$ has a finite relative generating system, 
then we say that $G$ is {\it finitely generated relative to} $\H$ or 
$(G,\H)$ is {\it relatively finitely generated}. 
When we put $\mathcal H=\bigsqcup_{H\in\H}(H\setminus \{1\})$, 
$X$ is a relative generating system of $(G,\H)$ 
if and only if $X\sqcup \mathcal H$ is a generating system of $G$.
Since $H\setminus \{1\}$ has the natural involution induced by inverse on $G$, 
$\mathcal H$ has the natural involution. 
If a relative generating system $X$ of $(G,\H)$ is symmetric, 
then $X\sqcup \mathcal H$ is symmetric. 
Throughout this note, we assume that any system is symmetric.  

We fix a generating system $X$ of $G$ relative to $\H$. 
Let us denote by $\left(X \sqcup \mathcal H\right)^\ast$ 
the free monoid generated by $X \sqcup \mathcal H$, 
that is, the set of words over $X \sqcup \mathcal H$. 
Then $\left(X \sqcup \mathcal H\right)^\ast$ has the natural involution. 
For two words $W_1, W_2 \in \left(X \sqcup \mathcal H\right)^\ast$, 
$W_1=_F W_2$ (resp. $W_1=_G W_2$) means that $W_1$ and $W_2$ represent 
the same element of $F$ (resp. $G$). 
For each word $W \in \left(X \sqcup \mathcal H\right)^\ast$, 
$\overline W$ denotes the element of $G$ represented by $W$
and $\| W \|$ means the word length of $W$. 
Note that $\left(X \sqcup \mathcal H\right)^\ast\ni W\mapsto \overline W\in G$ is surjective. 
A word $W\in \left(X \sqcup \mathcal H\right)^\ast$ is said to be {\it reduced} 
if $\|W \|\le \|W' \|$ for any word $W'\in \left(X \sqcup \mathcal H\right)^\ast$ 
such that $W=_F W'$. 
For any word $W'\in \left(X \sqcup \mathcal H\right)^\ast$, 
there exists the unique reduced word $W\in \left(X \sqcup \mathcal H\right)^\ast$ such that $W'=_F W$.  

For a subset $\mathcal R$ of $\left(X \sqcup \mathcal H\right)^\ast$, 
if the normal closure of the image of $\mathcal R$ in $F$ is equal to $\Ker \ \psi$, 
then the pair $(X,\mathcal R)$ is called a {\it relative presentation} of $(G,\H)$. 
Without loss of generality, 
we assume that $\mathcal R$ is reduced in the sense of \cite[Definition 2.24]{Osi06a}, 
that is, every $R\in \mathcal R$ is a reduced word in $\left(X \sqcup \mathcal H\right)^\ast$ 
and for each $R \in \mathcal R$, 
$\mathcal R$ contains $R^{-1}$ and all cyclic shifts of $R$ and $R^{-1}$. 
If both $X$ and $\mathcal R$ are finite, 
then a relative presentation $(X,\mathcal R)$ is said to be {\it finite}, 
and we say that $(G,\H)$ is {\it relatively finitely presented} 
or $G$ is {\it finitely presented relative to} $\H$. 
Throughout this note, 
any relative presentation is assumed to be reduced. 

We fix a finite relative presentation $(X,\mathcal R)$ of $(G,\H)$. 
For each $H\in \H$, we define $\Omega_H$ as the set of all elements of $H\setminus \{1\}$ 
such that each of them is a letter of some word $R\in \mathcal R$. 
Put $\Omega=\bigsqcup_{H\in\H}\Omega_H \subset \mathcal H$. 
Since $\mathcal R$ is reduced, 
$\Omega_H$ has the involution by inverse in $H$. 
Hence $\Omega$ also has the involution. 
Since $\mathcal R$ is finite, $\Omega$ is also finite. 
 
Let $W$ be any word of $\left(X \sqcup \mathcal H\right)^\ast$ 
with $\|W\|\leq n$ and $\overline W=1_G$. 
Then there exist $\{f_i\in \left(X \sqcup \mathcal H\right)^\ast\}_{i=1,\ldots,k}$ 
and $\{R_i\in \mathcal R\}_{i=1,\ldots, k}$ 
such that $W=_F \prod_{i=1}^k f_i^{-1} R_if_i$. 
We denote by $\area(W)$ the smallest number $k$ 
among all such representations of the image of $W$ in $F$ as above. 
If there exists the maximum of $\{\area(W) \mid \|W\|\leq n, \overline W=1_G\}$ 
for each non-negative integer $n$, then 
the {\it relative Dehn function} of $(G,\H)$ with respect to $(X,\mathcal R)$
is defined as 
\[
\delta^{rel}_{G,\H}:\{0\}\sqcup\N\ni n\mapsto 
\max\{\area(W) \mid \|W\|\leq n, \overline W=1_G\}\in \{0\}\sqcup\N.
\]
If not, then we say that 
the relative Dehn function of $(G,\H)$ with respect to $(X,\mathcal R)$
is not well-defined. 

We recall the definition of 
relatively hyperbolic groups by Osin (\cite[Definition 2.35]{Osi06a}).
\begin{Def}\label{O}
If $(G,\H)$ has a finite relative presentation $(X,\mathcal R)$ 
such that the relative Dehn function $\delta^{rel}_{G,\H}$ of $(G,\H)$ 
with respect to $(X,\mathcal R)$ is linear, 
we say that $G$ is {\it hyperbolic relative to} $\H$ or $(G,\H)$ is {\it relatively hyperbolic}. 
\end{Def}

This definition is independent of the choice of a finite relative presentation of $(G,\H)$ by \cite[Theorem 2.34]{Osi06a}. 

\begin{Rem}
Let $\K$ be a family of subgroups of $G$ such that $\K$ may contain trivial subgroups of $G$. 
We suppose that $(G,\K)$ has a finite relative presentation $(X,\mathcal R)$ 
and $G$ is hyperbolic relative to $\K$ in the sense of Osin. 
We denote by $\H$ a subfamily of $\K$ which consists of all non-trivial subgroups of $G$ in $\K$. 
Then $(X,\mathcal R)$ is also a finite relative presentation of $(G,\H)$ 
and $\delta^{rel}_{G,\H}=\delta^{rel}_{G,\K}$. 
Hence we do not consider families of subgroups containing trivial subgroups.
\end{Rem}

%%%%%%%%%%%%%%%%%%%%%%%%%%%%%%%%%%%%%%%
\subsection{Relative Cayley graphs and coned-off Cayley graphs}

Let $\Gamma$ be a graph whose set of vertices (resp.\ edges) 
is denoted by $V=V(\Gamma)$ (resp.\ $E=E(\Gamma)$). 
In this note, graphs mean oriented graphs and thus edges are oriented 
(see for example \cite{Ser80}). 
For each edge $e \in E$, we denote the origin and 
the terminus of $e$ by $e_-$ and $e_+$, respectively. 
Let us denote by $e^{-1}$ the inverse edge of $e$ whose origin (resp.\ terminus) 
is $(e^{-1})_-=e_+$ (resp.\ $(e^{-1})_+=e_-$). 

For each $i=1,2,\ldots,l$, let $e_i$ be an edge of $\Gamma$. 
If a sequence of edges $p=e_1e_2\cdots e_l$ satisfies 
$(e_i)_+=(e_{i+1})_-$ for each $i=1,2,\ldots, l -1$, 
then $p$ is called a {\it path} of $\Gamma$. 
The {\it length} of a path $p=e_1e_2\cdots e_l$ 
is defined as the number of its edges $l$ and denoted by $l(p)$. 
The vertices $(e_1)_-$ and $(e_l)_+$ are denoted by $p_-$ and $p_+$, respectively. 
For a path $p=e_1e_2\cdots e_l$, 
we define a path $p^{-1}$ as $e_l^{-1}e_{l-1}^{-1}\cdots e_1^{-1}$.
When $(e_1)_-=(e_l)_+$, the path $p$ is called a {\it cycle} of $\Gamma$. 
When $p$ is a cycle, a subpath of a cyclic shift of $p$ 
is also regarded as a subpath of $p$.
A path $p=e_1e_2\cdots e_l$ is called an {\it arc}
if $(e_1)_-,\ldots,(e_l)_-,(e_l)_+$ are mutually different. 
A cycle $p=e_1e_2\ldots e_l$ is a {\it circuit}
if $(e_1)_-,\ldots,(e_l)_-$ are mutually different
and $e_{1}\ne e_l^{-1}$ is satisfied. 

Let $Z$ be a set endowed with an involution $Z\ni z\mapsto z^{-1}\in Z$ and 
let $Z^\ast$ be the free monoid generated by $Z$ and be endowed with the induced involution. 
We call a map $\phi \colon E \to Z$ such that $\phi(e^{-1})=\phi(e)^{-1}$ a {\it labeling by} $Z$. 
The {\it label} of a path $p=e_1e_2 \cdots e_l$ of $\Gamma$ is defined as 
$\phi(p)=\phi(e_1)\phi(e_2) \cdots \phi(e_l) \in Z^\ast$. 
We note that $\phi(p^{-1})=\phi(p)^{-1}$ for any path $p$.

Let $X$ be a system of $G$, where $X$ is not necessarily a generating system of $G$. 
The {\it Cayley graph} $\Gamma(G,X)$ is a graph whose set of vertices (resp. edges) 
is $G$ (resp. $G\times X$) and for any edge $(g,x)\in G\times X$, 
two ends are $(g,x)_-=g$ and $(g,x)_+=gx$. 
Also we define the labeling by $X$ as $G\times X\ni (g,x)\mapsto x\in X$. 
The graph metric on $G$ is denoted by $d_X$, where 
if $g_1,g_2\in G$ cannot be connected 
by any path of $\Gamma (G,X)$, then we put $d_X(g_1,g_2)=\infty$. 

The {\it relative Cayley graph} 
$\overline \Gamma(G,\H,X)$ of $(G,\H,X)$ is the Cayley graph $\Gamma(G,X\sqcup \mathcal H)$.
We note that this has the labeling by $X\sqcup \mathcal H$. 

The {\it coned-off Cayley graph} $\widehat \Gamma(G,\H,X)$ is a graph 
constructed by adding cone-vertices and cone-edges to the Cayley graph $\Gamma(G,X)$. 
Here cone-vertices and cone-edges are defined as follows: 
For any left coset $gH\in G/H$ for every $H\in \H$, we define a {\it cone-vertex} named by $v(gH)$. 
For any $g\in G$ and any $H\in \H$, we define $[g,v(gH)]$ as a unique edge from $g$ to $v(gH)$
and its inverse edge $[v(gH),g]$ as a unique edge from $v(gH)$ to $g$. 
We call such edges {\it cone-edges}.  
For any $H\in\H$ and any different $g_1,g_2\in G$ with $g_1H=g_2H$, 
we call a path $[g_1, v(g_1H)][v(g_2H),g_2]$ the {\it cone-biedge} from $g_1$ to $g_2$. 
Note that $v(g_1H)=v(g_2H)$.

Note that the family $X$ is a relative generating system of $(G,\H)$ 
if and only if the relative (resp.\ coned-off) Cayley graph of $(G,\H,X)$ is connected. 

For a path $p$ of $\overline \Gamma(G,\H,X)$ and $H\in\H$, 
a subpath $q$ of $p$ is called an {$H$-component} of $p$ 
if $q$ is labeled by $(H\setminus \{1\})^\ast$ 
and not properly contained 
in any subpath $r$ of $p$ which is labeled by $(H\setminus \{1\})^*$.
We call a subpath $q$ of $p$ an {\it $\H$-component} 
if $q$ is an $H$-component of $p$ for some $H\in\H$. 
A path $p$ of $\overline \Gamma(G,\H,X)$ is said to be {\it locally minimal} 
if the length of any $\H$-component of $p$ is equal to $1$. 
A vertex of $p$ is called a {\it phase vertex} of $p$ if it is an end of an $\H$-component or
an edge labeled by $X$.
For paths $p$ and $p'$ of $\overline \Gamma(G,\H,X)$ and $H\in \H$, 
$H$-components $q$ of $p$ and $q'$ of $p'$ are said to be {\it connected} 
if vertices $q_-$ and $q'_-$ belong to the same left coset of $H$, that is, $(q_-)H=(q'_-)H$. 
For a path $p$ of $\overline \Gamma(G,\H,X)$ and $H\in \H$, 
an $H$-component $q$ of $p$ is said to be {\it isolated} 
if any other $H$-component $r$ of $p$ is not connected to $q$. 
A path $p$ of $\overline \Gamma(G,\H,X)$ is said to be {\it without backtracking} 
if every $\H$-component of $p$ is isolated. 

Let a path $p$ of $\widehat \Gamma(G,\H,X)$ consist 
of cone-biedges and edges labeled by elements of $X$. 
In particular its ends belong to $G$. 
For $H\in\H$ and $g\in G$, 
we say that $p$ {\it penetrates} the left coset $gH$ by a subpath $q$ of $p$ 
if $q$ consists of cone-biedges with a vertex $v(gH)$ 
and is not properly contained in any subpath $r$ of $p$ 
which consists of cone-biedges with a vertex $v(gH)$. 
Such a subpath $q$ of $p$ is called a {\it penetrating subpath} of $gH$. 
The path $p$ is said to be {\it locally minimal} 
if for any left coset $gH$, whenever $p$ penetrates $gH$ by a subpath $q$, 
the subpath $q$ consists of a cone-biedge of $\widehat \Gamma(G,\H,X)$. 
A vertex of $p$ is called a {\it phase vertex} of $p$ if it is an end of a penetrating subpath 
or an edge labeled by $X$.
The path $p$ is called a path {\it without backtracking} 
if every left coset $gH$ is penetrated by $p$ less than twice. 

Now we consider relation between paths of $\overline \Gamma(G,\H,X)$ and $\widehat \Gamma(G,\H,X)$. 
Note that $\Gamma(G,X)$ is contained in both $\overline \Gamma(G,\H,X)$ and $\widehat \Gamma(G,\H,X)$. 
Take an arbitrary path $p$ of $\overline \Gamma(G,\H,X)$.
We construct a path $\pi(p)$ of $\widehat \Gamma(G,\H,X)$ from $p$ 
by replacing all edges labeled by $\mathcal H$ with cone-biedges, that is, 
replacing an edge $(g, h)$ labeled by $h\in H\setminus \{1\}$ from $g\in G$ to $gh\in G$
with a cone-biedge $[g,v(gH)][v(gH),gh]$. 
Here we collect well-known properties between relative Cayley graphs and coned-off Cayley graphs. 

\begin{Lem}\label{dict}
Suppose that $G$ is finitely generated relative to $\H$ and denote by $X$ 
a finite generating system of $G$ relative to $\H$. 
Then we have the following: 
\begin{enumerate}
\item The map $id_G:G\to G$ is $(2,0)$-quasi-isometric from $G$ with $d_{X\sqcup \mathcal H}$ 
to $G$ with the distance induced by $d_{\widehat \Gamma(G,\H,X)}$. 
\item For any $g_1,g_2\in G$ we take any path $\widehat p$ from $g_1$ to $g_2$ 
of $\widehat \Gamma(G,\H,X)$ which consists of cone-biedges and edges labeled by elements of $X$. 
Then there exists a unique path $\overline p$ from $g_1$ to $g_2$ of 
$\overline \Gamma(G,\H,X)$ with $\pi(\overline p)=\widehat p$ 
and the set of phase vertices of $\widehat p$ is equal to that of $\overline p$. 
\item For any $g_1,g_2\in G$, we consider a path $\overline p$ from $g_1$ to $g_2$ 
of $\overline \Gamma(G,\H,X)$ and put $\widehat p=\pi(\overline p)$. 
Then the following hold:
\begin{enumerate} 
\item For a subpath $\overline q$ of $\overline p$, 
$\overline q$ is an $H$-component of $\overline p$
if and only if $\widehat p$ penetrates a left coset of $H$ 
by a penetrating subpath $\pi(\overline q)$. 
\item The path $\overline p$ is locally minimal if and only if $\widehat p$ is locally minimal. 
\item The path $\overline p$ is a path without backtracking if and only 
if $\widehat p$ is a path without backtracking. 
\item The path $\overline p$ is a locally minimal circuit without backtracking 
(resp. a locally minimal arc without backtracking)
if and only if $\widehat p$ is a circuit (resp. an arc).  
\end{enumerate}
\end{enumerate}
\end{Lem}

We recall the definition of fine graphs due to Bowditch \cite[p.9]{Bow12}. 

\begin{Def}
A graph $\Gamma$ is said to be {\it fine} 
if for any edge $e$ of $\Gamma$ and any positive integer $n$, 
the set of circuits of length at most $n$ containing $e$ in $\Gamma$ is finite. 
\end{Def}

We confirm well-known facts that both hyperbolicity and fineness 
are independent of the choice of a finite relative generating system. 
\begin{Lem}\label{choiceX}
Suppose that $(G,\H)$ has two finite relative generating systems $X$ and $Y$. 
Then we obtain the following: 
\begin{enumerate}
\item $X\sqcup Y$ is also a finite relative generating system of $(G,\H)$ 
and the natural inclusion from $\widehat{\Gamma}(G,\H,X)$ into $\widehat{\Gamma}(G,\H,X\sqcup Y)$
(resp.\ from $\widehat{\Gamma}(G,\H,Y)$ to $\widehat{\Gamma}(G,\H,X\sqcup Y)$) is a quasi-isometry. 
\item $\widehat{\Gamma}(G,\H,X)$ is hyperbolic if and only if $\widehat{\Gamma}(G,\H,Y)$ is hyperbolic. 
\item $\widehat{\Gamma}(G,\H,X)$ is fine if and only if $\widehat{\Gamma}(G,\H,Y)$ is fine. 
\end{enumerate}
\end{Lem}

\begin{proof}
(i) It is clear that $X\sqcup Y$ is a finite generating system of $G$ relative to $\H$. 
Each of $\widehat \Gamma(G,\H,X)$ and $\widehat \Gamma(G,\H,Y)$ 
is naturally regarded as a subgraph of $\widehat \Gamma(G,\H,X\sqcup Y)$. 
By \cite[Lemma 2.7]{MP-W11a}, the natural inclusion from $\widehat \Gamma(G,\H,X)$ 
into $\widehat \Gamma(G,\H,X\sqcup Y)$ 
(resp.\ from $\widehat \Gamma(G,\H,Y)$ into $\widehat \Gamma(G,\H,X\sqcup Y)$) 
is a quasi-isometry. 

(ii) Since $\widehat \Gamma(G,\H,X)$ is quasi-isometric to $\widehat \Gamma(G,\H,Y)$, 
one is hyperbolic if and only if the other is hyperbolic. 

(iii) Note that the stabilizer of any edge of $\widehat \Gamma(G,\H,X\sqcup Y)$ is finite. 
By the same way as in the proof of \cite[Lemma 2.9]{MP-W11a}, 
we can show that if $\widehat \Gamma(G,\H,X)$ is fine, 
then $\widehat \Gamma(G,\H,X\sqcup Y)$ is also fine. 
Since any subgraph of a fine graph is fine, 
the coned-off Cayley graph $\widehat \Gamma(G,\H,X)$ is fine 
if $\widehat \Gamma(G,\H,X\sqcup Y)$ is fine. 
\end{proof}

BCP property due to Farb \cite{Far98} can be extended 
to a group which is not necessarily finitely generated. 
The following is based on \cite[Theorem 2.14 (2)]{MP11}. 

\begin{Def}\label{BCP}
Suppose that $(G,\H)$ has a finite relative generating system $X$ of $(G,\H)$. 
We say that $(G,\H,X)$ satisfies {\it BCP property} 
if there exists a finite relative generating system $Y$ of $(G, \H)$ satisfying the following:
For any $\mu\ge 1$ and any $C\ge 0$, there exists a constant $a=a(\mu, C)$ such that 
for any $g\in G$, any $H\in \H$ and 
any $(\mu,C)$-quasigeodesics $p$ and $q$ without backtracking in $\widehat{\Gamma}(G,\H,X)$ 
satisfying $p_-,p_+,q_-,q_+ \in G$, $p_-=q_-$ and $p_+=q_+$, 
if $p$ penetrates $gH$ by a penetrating subpath $s$ and $q$ does not penetrate $gH$, 
then $d_Y(s_-,s_+)\le a$. 
\end{Def}

By Lemma \ref{dict}, we have the following equivalent definition. 
\begin{Def}\label{BCP'}
Suppose that $(G,\H)$ has a finite relative generating system $X$ of $(G,\H)$. 
Then $(G,\H,X)$ satisfies BCP property 
if there exists a finite relative generating system $Y$ of $(G, \H)$ satisfying the following:
For any $\mu\ge 1$ and any $C\ge 0$, there exists a constant $a=a(\mu, C)$ 
such that for any $H\in \H$ and any $(\mu, C)$-quasigeodesics $p$ and $q$ without backtracking 
in $\overline{\Gamma}(G,\H,X)$ satisfying $p_-=q_-$ and $p_+=q_+$, 
if an $H$-component $s$ of $p$ is connected 
with no $H$-components of $q$, then $d_Y(s_-,s_+)\le a$.
\end{Def}

\begin{Rem}
By \cite[Appendix A]{Dar03}, 
when $G$ is finitely generated, 
Definitions \ref{BCP} and \ref{BCP'} are equivalent to 
BCP property due to Farb \cite{Far98}. 
\end{Rem}

%%%%%%%%%%%%%%%%%%%%%%%%%%%%%%%%%%
\subsection{Condition $(a)$}

Before defining Condition $(a)$, we show the following lemma. 

\begin{Lem}\label{cycle} 
Let $X$ be a relative generating system of $(G,\H)$. 
Take a subfamily $\H'$ of $\H$. 
Then the following are equivalent:
\begin{enumerate}
\item Any circuit of $\widehat{\Gamma}(G,\H,X)$
penetrates no left cosets of each $H\in \H\setminus\H'$;
\item Any locally minimal circuit without backtracking of 
$\overline{\Gamma}(G,\H,X)$ contains no isolated $H$-components for each $H\in\H\setminus \H'$; 
\item Any cycle of $\overline{\Gamma}(G,\H,X)$ contains no isolated $H$-components 
with different endpoints for each $H\in\H\setminus \H'$; 
\item The group $G$ is freely decomposed into $\langle X,\H'\rangle \ast(\ast_{H\in \H\setminus\H'}H)$, where $\langle X,\H'\rangle$ is a subgroup generated by $X$ and all elements of $\H'$ in $G$;
\item Any different elements of each $H\in \H\setminus \H'$ 
cannot be connected by any path of $\overline{\Gamma}(G,\H,X)$ 
without edges belonging to $H\times (H\setminus \{1\})\subset G\times (H\setminus \{1\})$. 
\end{enumerate}
\end{Lem}

\begin{proof}
It follows from Lemma \ref{dict} (iii) (d) that (i) is equivalent to (ii). 
It is clear that (iii), (iv) and (v) are equivalent. 
Also (iii) obviously implies (ii).
 
We prove that (ii) implies (iii). 
Take an $H\in \H$ and a cycle $c$ of $\overline \Gamma(G,\H,X)$. 
Suppose that $c$ contains an isolated $H$-component $s$ with $s_-\ne s_+$. 
We consider an edge $e=(s_-, h)\in G\times (H\setminus \{1\})$ where $h=_F\phi(s)$. 
When we replace $s$ with $e$, we have a cycle $c'$, 
where $e$ is an isolated $H$-component of $c'$. 
There exist shortest cycles which 
contain $e$ as an isolated $H$-component.
Such cycles are locally minimal and without backtracking.    
\end{proof} 

The following is clear. 
\begin{Lem}\label{iso}
Let $X$ be a relative generating system of $(G,\H)$. 
Take a subfamily $\H'$ of $\H$. 
Suppose that $\H'$ satisfies 
one of five (mutually equivalent) conditions in Lemma \ref{cycle}.
Then the natural embedding from 
$\langle X,\H'\rangle$ with $d_{X\sqcup \mathcal H'}$
to $G$ with $d_{X\sqcup \mathcal H}$ is isometric, where 
$\mathcal H':=\bigsqcup_{H'\in \H'} (H'\setminus\{1\})$. 
\end{Lem}

Now we define Condition $(a)$. 
\begin{Def}\label{*}
Let $X$ be a relative generating system of $(G,\H)$. 
We say that $(G, \H, X)$ satisfies {\it Condition $(a)$} 
if a finite subfamily $\H'$ of $\H$ satisfying
one of five (mutually equivalent) conditions in Lemma \ref{cycle}.

We say that $(G, \H)$ satisfies {\it Condition $(a)$} 
if $(G, \H, X)$ satisfies {\it Condition $(a)$} for some 
relative generating system $X$ of $(G,\H)$. 
\end{Def}

\begin{Rem}\label{aut}
If $\H$ is finite, then $(G,\H, X)$ trivially satisfies Condition $(a)$
for any relative generating system $X$ of $(G,\H)$.
Also if $(G,\H)$ has a finite relative presentation $(X,{\mathcal R})$, 
then $(G,\H, X)$ satisfies Condition $(a)$.
\end{Rem}

The following is clear and is used without being explicitly mentioned
from now on. 
\begin{Lem}
Let $X$ be a relative generating system of $(G,\H)$. 
We denote by $\H_X$ the smallest subfamily of $\H$ satisfying
one of five (mutually equivalent) conditions in Lemma \ref{cycle}.
Then $(G, \H, X)$ satisfies Condition $(a)$ 
if and only if $\H_X$ is finite. 
\end{Lem}

\begin{Lem}\label{(*)}
Let $X$ and $X'$ be relative generating systems of $(G,\H)$, respectively.
Suppose that $X'$ includes $X$. 
Then $(G,\H, X)$ satisfies Condition $(a)$ if 
$(G,\H, X')$ satisfies Condition $(a)$.
Moreover suppose that $X'\setminus X$ is finite. 
Then $(G,\H, X)$ satisfies Condition $(a)$ only if 
$(G,\H, X')$ satisfies Condition $(a)$. 

In particular if $(G, \H)$ is relatively finitely generated 
and $(G, \H)$ satisfies Condition $(a)$,
then $(G, \H, X)$ satisfies Condition $(a)$ 
for any finite relative generating system $X$ of $(G,\H)$.
\end{Lem}

\begin{proof}
Since we have $\H_X\subset \H_{X'}$, 
$\H_X$ is finite if $\H_{X'}$ is finite. 

Suppose that $\H_X$ is finite and that $Y=X'\setminus X$ is finite. 
Since $Y$ is finite, 
there exists a finite subfamily $\H'$ of $\H$ such that $\H_{X}\subset \H'$ 
and all elements of $Y$ belong to 
$\langle X,\H_{X}\rangle \ast(\ast_{H\in \H'\setminus\H_X}H)=\langle X,\H'\rangle$. 
Since 
$G=\langle X,\H_{X}\rangle \ast(\ast_{H\in \H'\setminus\H_X}H)\ast(\ast_{H\in \H\setminus \H'}H)
=\langle X, \H'\rangle \ast (\ast_{H\in \H\setminus \H'}H)$, 
$Y$ has no relation with $H\in \H\setminus \H'$. 
Hence $\H'$ includes $\H_{X'}$
and thus $\H_{X'}$ is finite.  
\end{proof}

%%%%%%%%%%%%%%%%%%%%%%%%%%%%%%%%%%
\section{Relative hyperbolicity for groups}\label{geometry}

In this section we discuss definitions of relative hyperbolicity for groups.

\subsection{Equivalent definitions of relative hyperbolicity}
We show the following. 

\begin{Prop}\label{relhypeq}
Let $G$ be a group and $\H$ be a family of non-trivial subgroups of $G$.
Suppose that $(G,\H)$ has a finite relative generating system $X$. 
Then the following conditions are equivalent: 
\begin{enumerate}
\item[(i)] The coned-off Cayley graph $\widehat{\Gamma}(G,\H,X)$ is hyperbolic, fine 
and $(G,\H,X)$ satisfies Condition $(a)$; 
\item[(ii)] $(G,\H)$ has a finite relative presentation $(X,\mathcal R)$ 
whose relative Dehn function is linear; 
\item[(iii)] The relative Cayley graph $\overline{\Gamma}(G,\H,X)$ is hyperbolic 
and $(G,\H,X)$ satisfies BCP property and Condition $(a)$;
\item[(iv)] The family $\H$ is hyperbolically embedded in $G$ with respect to $X$
(in the sense of \cite{D-G-O12}) and $(G,\H,X)$ satisfies Condition $(a)$.  
\end{enumerate}
The family $\H$ is {\it hyperbolically embedded} in $G$ with respect to $X$ if 
$\overline{\Gamma}(G,\H,X)$ is hyperbolic and 
for any $H\in \H$ and any positive integer $n$, 
the set of vertices of $H$ connected to $1$ by paths with at most length $n$
in $\overline{\Gamma}(G,\H,X)$ without edges belonging to $H\times (H\setminus \{1\})$
is finite (see \cite[Theorem 4.24 a) and Definition 4.25]{D-G-O12} for details).   
\end{Prop}

\begin{Rem}
Recall that $(G,\H)$ is said to be relatively hyperbolic if 
there exists a finite relative presentation of $(G,\H)$ 
whose relative Dehn function is linear.
Also note that the condition (i) in Proposition \ref{relhypeq} 
is independent of choice of finite relative generating systems of $(G,\H)$ 
by Lemmas \ref{choiceX} and \ref{(*)}. 
Hence $(G,\H)$ is relatively hyperbolic if and only if 
$(G,\H)$ is finitely relatively generated and 
some (resp. any) finite relative generating system $X$ of $(G,\H)$ satisfies
one of (mutually equivalent) four conditions in Proposition \ref{relhypeq}.
\end{Rem}
 
If $G$ is countable and $\H$ is finite, it is well-known 
that (i), (ii) and (iii) are equivalent 
(see \cite[Appendix A]{Dar03}, \cite[Appendix]{Osi06a}, \cite[Section 3]{Hru10}, 
\cite[Theorem 2.14]{MP11} and \cite[Theorem 2.18]{Yan12}).  
Note that if $\H$ is finite, then $(G,\H,X)$ automatically satisfies Condition $(a)$. 

Even if $G$ and $\H$ are not necessarily countable nor finite, respectively, 
some implications follows from known facts.
Indeed since $(G,\H, X)$ satisfies Condition $(a)$ whenever 
$(X,\mathcal R)$ is a finite relative presentation of $(G,\H)$, 
it follows from \cite[Proposition 4.28 a)]{D-G-O12} that (ii) implies (iv) 
and also it follows from \cite[Theorem 2.14]{MP11} that (ii) implies (iii). 
We remark that \cite[Theorem 2.14]{MP11} also claims the following: 

\begin{Prop}\label{k-similar} 
Suppose that $(G,\H)$ has a finite relative presentation $(X,\mathcal R)$ 
whose relative Dehn function is linear. 
Then $(G,\H)$ has a finite relative generating system $Y\supset X$
satisfying the following: 
For any $\mu\ge 1$, $C\ge 0$, $k\ge 0$, there exists a constant $K=K(\mu, C, k)$ 
such that for any $(\mu,C)$-quasigeodesics $p$ and $q$ which are $k$-similar by $d_X$ 
and without backtracking in $\overline{\Gamma}(G,\H,X)$, 
the sets of phase vertices of $p$ and $q$ are contained 
in the closed $K$-neighborhoods with respect to $d_Y$ of each other. 
Here the paths $p$ and $q$ are said to be {\it $k$-similar} by $d_X$
if there exists $k\ge 0$ such that $d_X(p_-,q_-)\le k$ and $d_X(p_+,q_+)\le k$.
\end{Prop}

Also it follows from \cite[Proposition 4.28 b)]{D-G-O12} that (iv) implies (ii) 
in view of Lemma \ref{cycle} (vi).
Indeed when $\H$ is hyperbolically embedded in $G$ with respect to $X$, 
$\H_X$ is hyperbolically embedded in $G_0:=\langle X,\H_X\rangle$ 
because $\overline{\Gamma}(G_0, \H_X, X)$ 
is naturally isometrically embedded into $\overline{\Gamma}(G, \H, X)$
by Lemma \ref{iso}.
When $(G,\H,X)$ satisfies Condition $(a)$, $\H_X$ is finite. 
Hence \cite[Proposition 4.28 b)]{D-G-O12} implies that $(G_0,\H_X)$ has 
a finite relative presentation $(X,\mathcal R_0)$ 
whose relative Dehn function is linear. 
Then $(X,\mathcal R_0)$ can be regarded as 
a finite relative presentation of $(G,\H)$ by Lemma \ref{cycle} (iv).  
Then the relative Dehn function of $(G,\H)$ with respect to $(X,\mathcal R_0)$ is linear
because it is equal to 
the relative Dehn function of $(G_0,\H_X)$ with respect to $(X,\mathcal R_0)$.

In order to complete a proof of Proposition \ref{relhypeq}, 
we show that (i) implies (iv) and that (iii) implies (i). 

The following shows that (i) implies (iv). 
\begin{Lem}\label{locfinite}
Let $X$ be a generating system of $(G,\H)$.
Suppose that $\widehat{\Gamma}(G,\H,X)$ is fine.
Then for any $H\in \H$ and any positive integer $n$, 
the set of vertices of $H$ connected to $1$ by paths with length at most $n$
in $\overline{\Gamma}(G,\H,X)$ without edges belonging to $H\times (H\setminus \{1\})$
is finite.   
\end{Lem}
\begin{proof}
Take a positive integer $n$, $H\in \H$, $h\in H\setminus\{1\}$ and a shortest path $q$ from $h$ to $1$ 
in $\overline{\Gamma}(G,\H,X)$ without edges belonging to $H\times (H\setminus \{1\})$. 
Suppose that length of $q$ is at most $n$. 
We denote by $e$ the edge from $1$ to $h$ labeled by $h\in H\setminus \{1\}$. 
Then $p:=eq$ is a locally minimal circuit without backtracking in $\overline{\Gamma}(G,\H,X)$, 
whose length is at most $n+1$.
When we consider the cycle $\pi(p)$ in $\widehat{\Gamma}(G,\H,X)$ 
which is constructed on the above of Lemma \ref{dict}, 
this is a circuit by Lemma \ref{dict} (3)(d).
Note that $\pi(p)$ contains the cone-biedge $[1,v(H)][v(H),h]$ 
and has length at most $2(n+1)$.
When $\widehat{\Gamma}(G,\H,X)$ is fine, 
cardinality of circuits containing $[1,v(H)]$  
with length at most $2(n+1)$ is finite. 
Hence the assertion follows. 
\end{proof}

Finally we show that (iii) implies (i) (refer to \cite[A.2 Proposition 1]{Dar03}).
\begin{Lem}\label{circuit}
Suppose that $(G,\H)$ has a finite relative generating system $X$ and 
$(G, \H, X)$ satisfies BCP property. 
We take the finite relative generating system $Y$ of $(G,\H)$ in Definition \ref{BCP}. 
Let $\widehat{c}$ be a circuit of length $L\ge 1$ in $\widehat{\Gamma}(G,\H,X)$.
Then we obtain a cycle $c$ in $\Gamma(G, X\sqcup Y)$ 
by replacing any cone-biedge of $\widehat{c}$ with a path labeled by elements of $Y$ 
whose $d_Y$-length is at most $a=a(L,0)$. 
The length of $c$ is thus at most $aL$.
\end{Lem}

\begin{proof}
Since any circuit of length $L\le 2$ in $\widehat{\Gamma}(G,\H,X)$ 
does not contain any cone-biedge, we assume that $L\ge 3$. 

Let $\widehat{c}$ be a circuit of length $L$ in $\widehat{\Gamma}(G,\H,X)$.
Then this is a locally minimal path without backtracking. 
Since $L\ge 3$, the circuit $\widehat c$ contains two different vertices $u,v\in G$. 
We divide $\widehat c$ into two paths $\widehat p$ and $\widehat q$ with $\widehat c=\widehat p {\widehat q}^{-1}$, $\widehat p_-=\widehat q_-=u$ and $\widehat p_+=\widehat q_+=v$.
The paths $\widehat p$ and $\widehat q$ are locally minimal $(L,0)$-quasigeodesics without backtracking from $u$ to $v$. 
We hence obtain paths $p$ and $q$ in $\Gamma(G,X\sqcup Y)$ by replacing any cone-biedge of $\widehat p$ and $\widehat q$ with a path of  $d_Y$-length at most $a$ whose labels are elements of $Y$ because for any $H\in \H$, $\widehat p$ and $\widehat q$ penetrate no common cosets of $G/H$. 
Put $c=pq^{-1}$. 
The path $c$ is a cycle in $\Gamma(G,X\sqcup Y)$ whose length is at most $aL$. 
\end{proof}

\begin{Lem}\label{BCP&(*)}
Suppose that $(G,\H)$ has a finite relative generating system $X$.
If $(G,\H,X)$ satisfies BCP property and Condition $(a)$, 
then $\widehat{\Gamma}(G,\H,X)$ is fine.
\end{Lem}

\begin{proof}
Let $e$ be an edge of $\widehat{\Gamma}(G,\H,X)$ and let $n$ be a positive integer.
If $n\le 2$, then the set of circuits of length at most $n$ containing $e$ is finite 
because such a circuit does not contain any cone-biedge and $X$ is finite. 

We thus assume that $n\ge 3$. 
Let us denote by $Y$ the finite relative generating system of $(G,\H)$ in Definition \ref{BCP}. 
For any circuit $\widehat c$ in $\widehat \Gamma(G,\H,X)$ containing $e$ of length $n$, 
a cycle $c$ of length at most $a(n,0)\cdot n$ 
in $\Gamma(G,X\sqcup Y)$ is constructed by Lemma \ref{circuit}. 
Since two different vertices of $c$ 
may be connected by a cone-biedge in $\widehat{c}$, 
the cycle $c$ can be obtained from 
at most $_{a(n,0)\cdot n}\!C_2\cdot\#\H_X$ mutually different circuits 
in $\widehat \Gamma(G,\H,X)$ by the way of Lemma \ref{circuit}
in view of Lemma \ref{cycle} (1). 
Here we denote the cardinality of $\H_X$ by $\#\H_X$, 
which is finite by Condition $(a)$.

Hence the cardinality of the set of circuits containing $e$ of length $n$ 
is at most $\{\#(X\sqcup Y)\}^{a(n,0)\cdot n}\cdot _{a(n,0)\cdot n}\!C_2\cdot\#\H_X$, 
which is finite.  
\end{proof}

By Lemma \ref{dict} (i) and Lemma \ref{BCP&(*)}, we obtain that (iii) implies (i).

\subsection{The case where the family of subgroups is finite}
We deal with the case where $\H$ is finite.
The following is based on \cite[Definition 2]{Bow12}. 

\begin{Def}\label{bow}
Let $\H$ be a finite family of infinite subgroups of $G$. 
Suppose that $G$ acts on a connected fine hyperbolic graph 
$\Gamma$ with finite edge stabilizers and finitely many orbits of edges, 
and $\H$ is a set of representatives of conjugacy classes of infinite vertex stabilizers. 
Then $\Gamma$ is called a {\it $(G,\H)$-graph}.
\end{Def}

The following is well-known if $G$ is finitely generated
(see \cite[Appendix A]{Dar03}).
\begin{Prop}
Let $G$ be a group and let $\H$ be a finite family of infinite subgroups of $G$. 
Then $G$ is hyperbolic relative to $\H$ 
if and only if there exists a $(G,\H)$-graph.
\end{Prop}
\begin{proof} 
Note that $(G,\H)$ satisfies Condition $(a)$ because $\H$ is finite. 
By \cite[Proposition 4.3]{MP-W11a}, 
there exists a $(G,\H)$-graph if and only if 
there exists a finite generating system $X$ of $G$ relative to $\H$, 
the coned-off Cayley graph $\widehat \Gamma(G,\H,X)$ is a $(G,\H)$-graph.
The assertion follows from condition (i) in Proposition \ref{relhypeq}. 
\end{proof}

\begin{Rem}
Another definition of relative hyperbolicity which allows the case 
where $G$ can be uncountable is given by Gerasimov \cite{Ger09} as follows: 
The group $G$ is said to be {\it hyperbolic relative to} a family $\H$ of subgroups of $G$
in the sense of Gerasimov 
if $G$ acts $3$-properly and $2$-cocompactly on a compact Hausdorff space $X$
and $\H$ is a set of representatives of conjugacy classes of maximal parabolic subgroups.  
Note that $\H$ is known to be finite (\cite[Main Theorem]{Ger09}).
When we suppose that $\H$ is finite, 
existence of the above $X$ is equivalent to existence of a $(G,\H)$-graph 
according to \cite[Section 9.1]{Ger10} and \cite[Proposition 7.1.2]{G-P11}. 
\end{Rem}

%%%%%%%%%%%%%%%%%%%%%%%%%%%%%%%%%%

\section{Relative quasiconvexity for subgroups}\label{relqcdef}

In this section we study
relatively undistorted subgroups and relatively quasiconvex subgroups.
Theorems \ref{cap}, \ref{qc-undist} and \ref{4} are proved in Section 4.3.   

\subsection{Condition $(b)$ and relatively finitely generated subgroups}
In this section we study Condition $(b)$ (Definition \ref{**}) 
and relatively finitely generated subgroups. 
We give several lemmas which are used in Sections 4.2 and 4.3.

First we define relatively finitely generated subgroups.
\begin{Not} 
Take a subset $Y\subset G$. We put 
\begin{align*}
&\H_{L,Y}:=\{L\cap yHy^{-1} \ |\ H\in \H, y\in Y, L\cap yHy^{-1}\neq \{1\}  \},\\
&\mathcal H_{L,Y}:=\bigsqcup_{L\cap yHy^{-1}\in \H_{L,Y}} \{(L\cap yHy^{-1})\setminus \{1\}\}.
\end{align*}
\end{Not}

\begin{Def}\label{relfgdef} 
The subgroup $L$ is {\it finitely generated relative to} $\H$ in $G$ if 
there exists a finite subset $Y\subset G$ such that 
$L$ is finitely generated relative to $\H_{L,Y}$.
\end{Def}

\begin{Rem}\label{YY'}
On the above, 
if we take a finite subset $Y'\subset G$ such that $Y'\supset Y$, 
then $L$ is also finitely generated relative to $\H_{L,Y'}$.
\end{Rem}

\begin{Lem}\label{conjfg}
The subgroup $L$ is finitely generated relative to $\H$ in $G$ if and only if 
$gLg^{-1}$ is finitely generated relative to $\H$ in $G$ for every $g\in G$. 
\end{Lem}
\begin{proof}
Suppose that we have a finite subset $Y\subset G$ such that 
$L$ is finitely generated relative to $\H_{L,Y}$. 
Take a finite relative generating system $S$ of $(L, \H_{L,Y})$. 
For any $g\in G$, we have 
$\H_{gLg^{-1},gY}=\{gKg^{-1} \ |\ K\in \H_{L,Y}\}$.  
Then $gSg^{-1}$ is a finite relative generating system of $(gLg^{-1}, \H_{gLg^{-1},gY})$. 
\end{proof}

\begin{Lem}\label{cap'} 
Let $K$ be a subgroup of $G$. 
Suppose that both $K$ and $L$ 
satisfy Condition $(b)$ with respect to $\H$ in $G$. 
Then $K\cap L$ also satisfies Condition $(b)$ with respect to $\H$ in $G$. 
\end{Lem}
\begin{proof}
We take arbitrarily $y_1,y_2\in G$ with $(K\cap L)y_1 \cap (K\cap L)y_2=\emptyset$. 
Note that $(K\cap L)y_1 \cap (K\cap L)y_2=(Ky_1\cap Ky_2) \cap (Ly_1\cap Ly_2)=\emptyset$. 
If both $Ky_1\cap Ky_2\ne \emptyset$ and $Ly_1\cap Ly_2\ne \emptyset$ held, 
we would have $k\in K $ and $l\in L$ with $y_1=ky_2=ly_2$. 
This implies that $k=l \in K\cap L$ and contradicts 
$(K\cap L)y_1 \cap (K\cap L)y_2=\emptyset$. 
We thus obtain at least one of $Ky_1\cap Ky_2=\emptyset$ and $Ly_1\cap Ly_2=\emptyset$. 

Without loss of generality, we assume that $Ky_1\cap Ky_2=\emptyset$ holds. 
It follows that $\H_{K\cap L,y_1,y_2}$ is finite 
because $\H_{K\cap L,y_1,y_2} \subset \H_{K,y_1,y_2}$.  
\end{proof}

The following is often useful. 
\begin{Lem}\label{Y}
Suppose that $L$ is finitely generated relative to $\H$ in $G$. 
Then there exists a finite subset $Y\subset G$ 
such that $L$ is finitely generated relative to $\H_{L,Y}$
and all elements of $Y$ belong to mutually different right cosets of $L$ in $G$.
\end{Lem}
\begin{proof}
Take a finite subset $Y\subset G$ such that 
$L$ is finitely generated relative to $\H_{L,Y}$ 
and also a finite relative generating system $S$ of $(L,\H_{L,Y})$.
If we have $y'=ly$ for some $y,y'\in Y$ and $l\in L$, 
then $S\sqcup\{l\}\sqcup\{l^{-1}\}$ is a finite relative generating system
of $(L,\H_{L,Y\setminus \{y'\}})$.
\end{proof}

The following implies that Condition $(b)$ is not strong. 
\begin{Lem}\label{weakrelfg}
Let $(G,\H)$ have a finite relative generating system $X$. 
Suppose that $(G,\H,X)$ satisfies Condition $(a)$.
If $L$ is finitely generated relative to $\H$ in $G$, 
then $L$ satisfies Condition $(b)$ with respect to $\H$ in $G$.
\end{Lem}
\begin{proof}
Take $y_1,y_2\in G$ such that $Ly_1\neq Ly_2$.
We take a finite subset $Y\subset G$ and 
a finite generating system $S$ of $L$ relative to $\H_{L,Y}$.
We can assume that 
$y_1$ and $y_2$ belong to $Y$ and 
all elements of $Y$ represent
mutually different right cosets of $L$ in $G$ by Remark \ref{YY'} and Lemma \ref{Y}. 
Suppose that $H$ is an element of $\H_{L,y_1,y_2}$. 
Then we take an element $l=y_1hy_2^{-1}\in L\cap y_1Hy_2^{-1}$ 
which is one of the nearest vertices from $1$ by $d_{S\sqcup \mathcal H_{L,Y}}$
among $L\cap y_1Hy_2^{-1}$. 
Take a geodesic $p$ of $\overline{\Gamma}(L,\H_{L,Y},S)$ from $l$ to $1$.
We note that $p$ is trivial if $l=1$. 
When we replace each edge $e$ of $p$ labeled by some $yh'y^{-1}\in (L\cap yH'y^{-1})\setminus \{1\}$ 
for some $L\cap yH'y^{-1}\in \H_{L,Y}$ with a path of three edges labeled by 
$y\in Y$, $h'\in H'\setminus \{1\}$ and $y^{-1}\in Y^{-1}$, we have a path $p'$ 
of $\overline{\Gamma}(G,\H,X\sqcup Y\sqcup Y^{-1}\sqcup S)$. 
Now we take the path $q$ from $1$ to $l$ of three edges $e_1$, $e_2$ and $e_3$ 
labeled by $y_1\in Y$, $h\in H$ and $y_2^{-1}\in Y^{-1}$, respectively. 
Then $qp'$ is a cycle in $\overline{\Gamma}(G, \H, X\sqcup Y\sqcup Y^{-1}\sqcup S)$. 
If $e_2$ is an isolated $H$-component in $qp'$, then 
$H\in \H_{X\sqcup Y\sqcup Y^{-1}\sqcup S}$ by Lemma \ref{cycle} (iii).
We note that $e_2$ satisfies the case if $l=1$.
If not, then we have an $H$-component $e_4$ connected to $e_2$
such that there exists no $H$-component
which is connected to $e_2$ on the subpath $r$ of $qp'$ 
from $(e_2)_+$ to $(e_4)_-$. 
Assume that $(e_2)_+=(e_4)_-$. 
Since all elements of $Y$ belong to mutually different right cosets of $L$ in $G$, 
we have an edge $e_5$ labeled by $y_2^{-1}$ such that $e_3^{-1}e_4e_5$ is a subpath 
of $p'$ from $l$ to some $l'\in L$. 
Since we have a path $e_1e_2e_4e_5$ of 
$\overline{\Gamma}(G, \mathbb{H},X\sqcup Y\sqcup Y^{-1}\sqcup S)$ 
from $1$ to $l'$, 
the element $l'$ lie in $y_1 H y_2^{-1}$.
Then the subpath $s$ of $p$ from $l'$ to $1$ is shorter than $p$. 
This contradicts choice of $l$. 
Hence we have $(e_2)_+\neq (e_4)_-$. 
When we denote by $e_6$ the edge from $(e_4)_-$ to $(e_2)_+$ 
labeled by an element of $H\setminus \{1\}$, 
$e_6r$ is a cycle, 
where $e_6$ is an isolated $H$-component with different ends. 
Hence $H\in \H_{X\sqcup Y\sqcup Y^{-1}\sqcup S}$ by Lemma \ref{cycle} (iii).
Thus $\H_{L,y_1,y_2}\subset \H_{X\sqcup Y\sqcup Y^{-1}\sqcup S}$. 
Since $\H_{X\sqcup Y\sqcup Y^{-1}\sqcup S}$ is finite by Lemma \ref{(*)}, 
$\H_{L,y_1,y_2}$ is also finite. 
\end{proof}

\begin{Lem}\label{2''}
Let $(G,\H)$ have a finite relative generating system $X$. 
Suppose that $(G,\H,X)$ satisfies Condition $(a)$.
Let $L$ be finitely generated relative to $\H$ in $G$. 
For a finite subset $Y\subset G$ such that 
all elements of $Y$ belong to mutually different right cosets of $L$ in $G$, 
when $S$ is a finite relative generating system of $(L,\H_{L,Y})$, 
$(L, \H_{L,Y}, S)$ satisfies Condition $(a)$.  
\end{Lem}
\begin{proof}
We take a finite relative generating system $S$ of $(L, \H_{L,Y})$. 
Take an element $L\cap y_0H_0y_0^{-1}\in (\H_{L,Y})_S$ where $y_0\in Y$ and $H_0\in\H$. 
We take a locally minimal circuit $c$ without backtracking 
in $\overline{\Gamma}(L, \H_{L,Y}, S)$ which 
have an edge labeled by some element of $(L\cap y_0H_0y_0^{-1})\setminus\{1\}$.
We suppose that the length of $c$ is shortest among such circuits.  
Then $c$ has an edge $e_0$ labeled by $y_0h_0y_0^{-1}\in (L\cap y_0H_0y_0^{-1})\setminus\{1\}$.
When we replace each edge in $c$ labeled by some element $yhy^{-1}\in (L\cap yHy^{-1})\setminus \{1\}$ 
for each $L\cap yHy^{-1}\in (\H_{L,Y})_S$ 
with a path of three edges labeled by $y\in Y$, $h\in H\setminus\{1\}$ and $y^{-1}\in Y^{-1}$, 
we have a cycle $c'$ in $\overline{\Gamma}(G, \H, X\sqcup Y\sqcup Y^{-1}\sqcup S)$.  
Suppose that $e_0$ is replaced with a path $e_1e_2e_3$ of three edges 
labeled by $y_0\in Y$, $h_0\in H_0\setminus\{1\}$ and $y_0^{-1}\in Y^{-1}$.
If $e_2$ is an isolated $H_0$-component in $c'$, then 
$H_0\in \H_{X\sqcup Y\sqcup Y^{-1}\sqcup S}$ by Lemma \ref{cycle} (iii).
If not, then we take an $H_0$-component $e$ connected to $e_2$
such that there exists no $H_0$-component which is connected to $e_2$
on the subpath $p$ of $c'$ from $e_+$ to $(e_2)_-$. 
Since all elements of $Y$ belong to 
mutually different right cosets of $L$ in $G$, 
we have $e_+\neq (e_2)_-$ by choice of $c$ 
(refer to the argument in Proof of Lemma \ref{weakrelfg}). 
When we denote by $e'$ the edge from $(e_2)_-$ to $e_+$ 
labeled by an element of $H_0\setminus \{1\}$, 
$e'p$ is a cycle, 
where $e'$ is an isolated $H_0$-component with different ends. 
Hence $H_0\in \H_{X\sqcup Y\sqcup Y^{-1}\sqcup S}$ by Lemma \ref{cycle} (iii).
Since $\H_{X\sqcup Y\sqcup Y^{-1}\sqcup S}$ is finite by Lemma \ref{(*)} and $Y$ is finite, 
$(\H_{L,Y})_S$ is also finite.
\end{proof}

We consider a kind of reduced version of $\H_{L,Y}$. 
\begin{Not} 
When $n$ is a positive integer and $Y=\{y_1,\ldots,y_n\}$ is a subset of $G$, 
we put
\begin{align*}
&\H_{L,Y}^r:=\{L\cap y_jHy_j^{-1} \in \H_{L,Y} \ |\ \text{for any }i<j, 
L\cap y_iHy_j^{-1}= \emptyset\},\\
&\mathcal H_{L,Y}^r:=\bigsqcup_{L\cap yHy^{-1}\in \H_{L,Y}^r} \{(L\cap yHy^{-1})\setminus \{1\}\}.
\end{align*}

When we have two finite subsets $Y=\{y_1,\ldots,y_{n}\}$ and 
$Y'=\{y_1,\ldots, y_n,y_{n+1}\ldots,y_{n'}\}$ of $G$, 
we have $\H_{L,Y}^r\subset \H_{L,Y'}^r$ by definition.   
\end{Not}

The following is trivial, but it is often useful. 
\begin{Lem}\label{teq'}
Take $y,y'\in G$ such that $Ly\neq Ly'$ and $H\in \H_{L,y,y'}$. 
Then for any $l\in L\cap yHy'^{-1}$, we have 
$L\cap yHy^{-1}=l(L\cap y'Hy'^{-1})l^{-1}$ and $L\cap yH=l(L\cap y'H)$.
\end{Lem}

\begin{Lem}\label{red'}
Let $(G,\H)$ have a finite relative generating system $X$. 
Suppose that $(G,\H,X)$ satisfies Condition $(a)$.
Let $L$ be finitely generated relative to $\H$ in $G$.  
Take any finite subset $Y=\{y_1,\ldots,y_n\}$ of $G$ such that 
$L$ is finitely generated relative to $\H_{L,Y}$ and 
all elements of $Y$ belong to mutually different right cosets of $L$ in $G$.  
Then $\H_{L,Y}\setminus \H_{L,Y}^r$ consists of finitely many elements and 
$L$ is finitely generated relative to $\H_{L,Y}^r$. 
Moreover for any finite relative generating system $S$ of $(L,\H^r_{L,Y})$, 
the identity map on $L$ is quasi-isometric from 
$(L,d_{S\sqcup \mathcal H^r_{L,Y}})$ to $(L,d_{S\sqcup \mathcal H_{L,Y}})$.  
Also $(L,\H_{L,Y}^r, S)$ satisfies Condition $(a)$. 
\end{Lem}
\begin{proof}
It follows from Lemma \ref{weakrelfg}
that $\H_{L,Y}\setminus \H_{L,Y}^r$ is a finite subset of $\H_{L,Y}$.
When $i,j\in\{1,\ldots,n\}$ and $H\in \H$ satisfy 
$L\cap y_iHy_j^{-1}\neq \emptyset$ and $i<j$, 
we take an element $l_{j,i,H}\in L\cap y_iHy_j^{-1}$.
We denote the subset of $L$ consisting of such elements and their inverses by $T$.
We fix a finite relative generating system $U$ of $(L,\H_{L,Y})$. 
Then $S':=T\sqcup U$ is a finite relative generating system of $(L,\H_{L,Y}^r)$ 
in view of Lemma \ref{teq'}.
Also we can straightforwardly prove that the identity map on $L$ is quasi-isometric from 
$(L,d_{U\sqcup \mathcal H_{L,Y}})$ to $(L,d_{S'\sqcup \mathcal H_{L,Y}^r})$
in view of Lemma \ref{teq'}. 
Hence Lemma \ref{choiceX} (i) implies that for any finite relative 
generating system $S$ of $(L,\H_{L,Y}^r)$, 
the identity map on $L$ is quasi-isometric from 
$(L,d_{S\sqcup \mathcal H_{L,Y}})$ to $(L_{S\sqcup \mathcal H_{L,Y}^r})$.
It follows from Lemma \ref{2''} that $(L,\H_{L,Y}^r)$ satisfies Condition $(a)$. 
\end{proof}

\begin{Lem}\label{red''}
Let $(G,\H)$ have a finite relative generating system $X$. 
Suppose that $(G,\H,X)$ satisfies Condition $(a)$.
Let $L$ be finitely generated relative to $\H$ in $G$.  
Take any finite subset $Y'=\{y_1,\ldots,y_{n'}\}$ of $G$ such that
all elements of $Y'$ belong to mutually different right cosets of $L$ in $G$. 
Take a subset $Y:=\{y_1,\ldots,y_{n}\}\subset Y'$
and suppose that 
$L$ is finitely generated relative to $\H_{L,Y}$.    
Then $\H_{L,Y'}^r\setminus \H_{L,Y}^r$ is finite.
\end{Lem}
\begin{proof}
Note that we have $\H_{L,Y}^r\subset \H_{L,Y'}^r$.
We take a finite relative generating system $S$ of $(L,\H_{L,Y}^r)$.
Then $S$ is also a finite relative generating system of $(L,\H_{L,Y'}^r)$.
By definition, we have $(\H_{L,Y}^r)_S\subset (\H_{L,Y'}^r)_S$.
By the condition (iv) of Lemma \ref{cycle}, we have 
$L=\langle S,(\H_{L,Y}^r)_S\rangle\ast(\ast_{K\in \H_{L,Y}^r\setminus(\H_{L,Y}^r)_S}K)
=\langle S,(\H_{L,Y'}^r)_S\rangle\ast(\ast_{K'\in \H_{L,Y'}^r\setminus(\H_{L,Y'}^r)_S}K')$
and thus we have 
$\H_{L,Y}^r\setminus(\H_{L,Y}^r)_S\supset\H_{L,Y'}^r\setminus(\H_{L,Y'}^r)_S$.
Since $(\H_{L,Y}^r)_S\subset (\H_{L,Y'}^r)_S$, 
we have $\H_{L,Y}^r\setminus(\H_{L,Y'}^r)_S\supset\H_{L,Y'}^r\setminus(\H_{L,Y'}^r)_S$.
In fact $\H_{L,Y}^r\setminus(\H_{L,Y'}^r)_S=\H_{L,Y'}^r\setminus(\H_{L,Y'}^r)_S$
by $\H_{L,Y}^r\subset \H_{L,Y'}^r$.
The assertion follows from the fact that 
$(L,\H_{L,Y'}^r)$ satisfies Condition $(a)$ by Lemma \ref{red'}. 
\end{proof}

\subsection{Relatively undistorted subgroups}
In this section we study relatively undistorted subgroups. 
We give several lemmas which are used in Section 4.3.

We define relatively undistorted subgroups.
\begin{Def}\label{rund}
Suppose that $(G,\H)$ has a finite relative generating system $X$.  
Then $L$ is {\it undistorted relative to} $\H$ in $G$ if 
$L$ is finitely generated relative to $\H$ in $G$ and 
for some finite subset $Y\subset G$ and some finite generating system $S$ of $(L,\H_{L,Y})$, 
the natural embedding from 
$(L,d_{S\sqcup {\mathcal H}_{L,Y}})$ to $(G,d_{X\sqcup {\mathcal H}})$ is quasi-isometric.

Also $L$ is {\it strongly undistorted relative to} $\H$ in $G$ if 
$L$ is finitely generated and 
for some finite generating system $S$ of $L$, 
the natural embedding from 
$(L,d_S)$ to $(G,d_{X\sqcup {\mathcal H}})$ is quasi-isometric.
\end{Def}

\begin{Rem}
The above is independent of choice of $X$ and $S$ by Lemma \ref{choiceX}. 
Also Lemmas \ref{Y'}, \ref{choice} and \ref{choiceY} claim that 
the above is in some sense independent of choice of $Y$.

If $L$ is strongly undistorted relative to $\H$ in $G$, 
then  $L$ is undistorted relative to $\H$ in $G$ by taking $Y$ as $\emptyset$. 
\end{Rem}

\begin{Ex}\label{Q}
Suppose that 
$G$ is finitely presented relative to $\H$.
We take a finite relative presentation $(X,{\mathcal R})$ 
and denote $Q$ the subgroup generated by $X$ and $\Omega$.
Since $\Omega$ is finite, $Q$ is finitely generated. 
Then $Q$ is undistorted relative to $\H$ in $G$ by \cite[Proposition 2.49]{Osi06a}.
\end{Ex}

\begin{Lem}\label{conjundist}
Suppose that 
$(G,\H)$ has a finite relative generating system $X$.
The subgroup $L$ is (strongly) undistorted relative to $\H$ in $G$ if and only if 
$gLg^{-1}$ is (strongly) undistorted relative to $\H$ in $G$ for every $g\in G$. 
\end{Lem}
\begin{proof}
We consider the setting in Proof of Lemma \ref{conjfg}.

Take $g\in G$. 
Then for any $l\in L$, we have 
\begin{align*}
&d_{\mathcal H\sqcup X\sqcup\{g\}\sqcup\{g^{-1}\}}(1,glg^{-1})\\
\le&d_{\mathcal H\sqcup X\sqcup\{g\}\sqcup\{g^{-1}\}}(1,g)+
d_{\mathcal H\sqcup X\sqcup\{g\}\sqcup\{g^{-1}\}}(g,gl)+
d_{\mathcal H\sqcup X\sqcup\{g\}\sqcup\{g^{-1}\}}(gl,glg^{-1})\\
\le&d_{\mathcal H\sqcup X\sqcup\{g\}\sqcup\{g^{-1}\}}(1,g)+
d_{\mathcal H\sqcup X\sqcup\{g\}\sqcup\{g^{-1}\}}(1,l)+
d_{\mathcal H\sqcup X\sqcup\{g\}\sqcup\{g^{-1}\}}(1,g^{-1})\\
\le&d_{\mathcal H\sqcup X\sqcup\{g\}\sqcup\{g^{-1}\}}(1,l)+2
\end{align*}
and we have by a similar way
\[
d_{\mathcal H\sqcup X\sqcup\{g\}\sqcup\{g^{-1}\}}(1,l)-2\le d_{\mathcal H\sqcup X\sqcup\{g\}\sqcup\{g^{-1}\}}(1,glg^{-1}).
\]
Also we have 
\[
d_{\mathcal H_{gLg^{-1},gY}\sqcup gSg^{-1}}(1,glg^{-1})
= d_{\mathcal H_{L,Y}\sqcup S}(1,l). 
\]
Since the identity from $(G, d_{\mathcal H\sqcup X})$
to $(G, d_{\mathcal H\sqcup X\sqcup\{g\}\sqcup\{g^{-1}\}})$ is quasi-isometric
by Lemmas \ref{choiceX} (i) and \ref{dict} (i), 
the assertion follows. 
\end{proof}

\begin{Lem}\label{Y'}
Suppose that 
$(G,\H)$ has a finite relative generating system $X$
and that $L$ is undistorted relative to $\H$ in $G$. 
We take a finite subset $Y'\subset G$ 
and a finite relative generating system $S'$ of $(L,\H_{L,Y'})$
such that the natural embedding 
$(L,d_{S'\sqcup {\mathcal H}_{L,Y'}})\to (G,d_{X\sqcup {\mathcal H}})$
is quasi-isometric. 
Then we have a subset $Y$ of $Y'$ and 
a finite relative generating system $S$ of $(L,\H_{L,Y})$
such that all elements of $Y$ belong to mutually different right cosets of $L$ in $G$ 
and the natural embedding 
$(L,d_{S\sqcup {\mathcal H}_{L,Y}})\to (G,d_{X\sqcup {\mathcal H}})$
is quasi-isometric.
\end{Lem}
\begin{proof}
Take $y'\in Y'$. 
If there exist $l\in L$ and $y\in Y'\setminus \{y'\}$ such that $y'=ly$, 
then we consider $Y'':=Y'\setminus \{y'\}$ and $S'':=S'\sqcup\{l\}\sqcup \{l^{-1}\}$. 
We replace each edge in $\overline{\Gamma}(L, \H_{L,Y'}, S')$ labeled by
$y'hy'^{-1}\in L\cap y'Hy'^{-1}$ for some $L\cap y'Hy'^{-1}\in \H_{L,Y'}$ 
with a path in 
$\overline{\Gamma}(L, \H_{L,Y''}, S'')$ 
of three edges labeled by $l$, $yhy^{-1}$ and $l^{-1}$.
Then the natural embedding 
$(L,d_{S''\sqcup {\mathcal H}_{L,Y''}})\to (G,d_{X\sqcup {\mathcal H}})$
is quasi-isometric.
By continuing the procedure, we have $Y$ and $S$ satisfying the assertion. 
\end{proof}

\begin{Lem}\label{choice}
Let $(G,\H)$ have a finite relative generating system $X$. 
Let $L$ be undistorted relative to $\H$ in $G$. 
If there exist a finite subset $Y\subset G$ and 
a finite generating system $S$ of $L$ relative to $\H_{L,Y}$ such that 
the natural embedding 
$(L,d_{S\sqcup {\mathcal H}_{L,Y}})\to (G,d_{X\sqcup {\mathcal H}})$
is quasi-isometric, then for any finite subset $Y'\subset G$ such that $Y\subset Y'$, 
the natural embedding 
$(L,d_{S\sqcup {\mathcal H}_{L,Y'}})\to (G,d_{X\sqcup {\mathcal H}})$
is quasi-isometric.
\end{Lem}
\begin{proof}
By Lemma \ref{choiceX}, the identity map from 
$(G,d_{X\sqcup {\mathcal H}})$ to $(G,d_{X\sqcup Y'\sqcup Y'^{-1}\sqcup S\sqcup \mathcal H})$
is quasi-isometric. 
We claim that the embedding from $(L, d_{S\sqcup {\mathcal H}_{L,Y'}})$
to $(G, d_{X\sqcup Y'\sqcup Y'^{-1}\sqcup S\sqcup {\mathcal H}})$ 
is large-scale Lipschitz. 
Indeed for any path $p$ of $\overline{\Gamma}(L, \H_{L,Y'},S)$, 
when we replace any edge labeled by some element 
$y'hy'^{-1}\in (L\cap y'Hy'^{-1})\setminus \{1\}$
with a path of three edges labeled by $y'$, $h$ and $y'^{-1}$, 
we have a path $p'$ of $(G, d_{X\sqcup Y'\sqcup Y'^{-1}\sqcup S\sqcup {\mathcal H}})$. 
Then length of $p'$ is at most three times of length of $p$.  
Also the identity map from 
$(L, d_{S\sqcup {\mathcal H}_{L,Y}})$ to $(L, d_{S\sqcup {\mathcal H}_{L,Y'}})$ 
is large-scale Lipschitz by $\H_{L,Y}\subset \H_{L,Y'}$. 
Since $(L,d_{S\sqcup {\mathcal H}_{L,Y}})\to (G,d_{X\sqcup {\mathcal H}})$ 
is quasi-isometric, 
$(L,d_{S\sqcup {\mathcal H}_{L,Y'}})\to (G,d_{X\sqcup {\mathcal H}})$
is quasi-isometric. 
\end{proof}

\begin{Lem}\label{choiceY}
Let $(G,\H)$ have a finite relative generating system $X$. 
Let $L$ be undistorted relative to $\H$ in $G$. 
Take a finite subset $Y\subset G$ and 
a finite relative generating system $S$ of $(L,\H_{L,Y})$ such that 
the natural embedding 
$(L,d_{S\sqcup {\mathcal H}_{L,Y}})\to (G,d_{X\sqcup {\mathcal H}})$
is quasi-isometric. 
When we take $l_y\in L$ for each $y\in Y$, we put $Y':=\{l_yy\}_{y\in Y}$.  
Then the natural embedding 
$(L,d_{S\sqcup {\mathcal H}_{L,Y'}})\to (G,d_{X\sqcup {\mathcal H}})$
is quasi-isometric.
\end{Lem}
\begin{proof}
We can easily confirm that 
$\overline{\Gamma}(L,\H_{L,Y'},S\sqcup\{l_y\}_{y\in Y}\sqcup\{l_y^{-1}\}_{y\in Y})$ and 
$\overline{\Gamma}(L,\H_{L,Y},S\sqcup\{l_y\}_{y\in Y}\sqcup\{l_y^{-1}\}_{y\in Y})$ are 
quasi-isometric.
\end{proof}

The following gives a characterization of relatively undistorted subgroups.   
\begin{Prop}\label{rundqc}
Let $(G,\H)$ have a finite relative generating system $X$. 
Suppose that $(G,\H,X)$ satisfies Condition $(a)$.
The subgroup $L$ is undistorted relative to $\H$ in $G$
if and only if $L$ satisfies Condition $(b)$
and there exist constants 
$\mu\ge 1$, $C\ge 0$ 
and a finite subset $Y\subset G$ satisfying the following:
for any $l\in L\setminus \{1\}$, there exists 
a locally minimal $(\mu, C)$-quasigeodesic $p$ 
in $\overline{\Gamma}(G,\H,X)$ without backtracking from $1$ to $l$ 
such that all vertices on $p$ are contained in $L\cup LY$.  
\end{Prop}
\begin{proof}
Suppose that $L$ is undistorted relative to $\H$ in $G$.
Since $L$ is finitely generated relative to $\H$ in $G$, 
$L$ satisfies Condition $(b)$ with respect to $\H$ in $G$ by Lemma \ref{weakrelfg}.
We take a finite set $Y\subset G$ 
and a relative generating system $S$ of $(L,\H_{L,Y})$
such that all elements of $Y$ belong to 
mutually different right cosets of $L$ in $G$. 
We take $l\in L\setminus \{1\}$ and 
a geodesic $p$ from $1$ to $l$ in $\overline{\Gamma}(L,\H_{L,Y},S)$. 
When we replace each edge in $p$ labeled by $yhy^{-1}\in (L\cap yHy^{-1})\setminus \{1\}$ 
to a path of three edges labeled by $y\in Y$, $h\in H\in \H$ and $y^{-1}\in Y^{-1}$
for any $L\cap yHy^{-1}\in \H_{L,Y}$, 
we have a path $p'$ in $\overline{\Gamma}(G, \H, X\sqcup Y\sqcup Y^{-1}\sqcup S)$.
Since $L$ is undistorted relative to $\H$ in $G$, there exists a pair of constants 
$(\mu, C)$ (which is independent of $l$ and $p$)
such that $p'$ is a $(\mu,C)$-quasigeodesic from $1$ to $l$. 
We can take a locally minimal $(\mu,C)$-quasigeodesic $p''$ without backtracking
from $1$ to $l$ such that all vertices on $p''$ belongs to $p'$. 
Since all vertices on $p'$ belong to $L\cup LY$, 
all vertices on $p''$ also belong to $L\cup LY$.

Suppose that
$L$ satisfies Condition $(b)$ 
and there exist constants $\mu\ge 1$, $C\ge 0$ 
and a finite subset $Y\subset G$ satisfying the following:
for any $l\in L\setminus \{1\}$, there exists 
a locally minimal $(\mu, C)$-quasigeodesic $p$ in $\overline{\Gamma}(G,\H,X)$ 
without backtracking from $1$ to $l$ 
such that all vertices on $p$ are contained in $L\cup LY$. 
We can assume that $Y$ contains $1$ and 
all elements of $Y$ belong to 
mutually different right cosets of $L$ in $G$. 
First we prove that $L$ is finitely generated relative to $\H_{L,Y}$.
We put 
\[
W_1:=\{yxy'^{-1} \ |\ y\in Y, x\in X, y'\in Y, yxy'^{-1}\in L\}.
\] 
For $y,y'\in Y$ and $H\in \H$ such that 
$y\neq y'$ and $L\cap yHy'^{-1}\neq \emptyset$, 
we fix an element $l_{y,y',H}\in L\cap yHy'^{-1}$. 
Note that $L\cap yHy'^{-1}=(L\cap yHy^{-1})l_{y,y',H}$ by Lemma \ref{teq'}.
We put 
\[
W_2:=\{l_{y,y',H} \ | \ y,y'\in Y, H\in \H, y\neq y' , L\cap yHy'^{-1}\neq \emptyset\}, 
\]
which is finite by Condition $(b)$.
We take any $l\in L\setminus \{1\}$ and  
any locally minimal $(\mu, C)$-quasigeodesic $p$ without backtracking 
of $\overline{\Gamma}(G,\H,X)$ from $1$ to $l$ 
such that all vertices on $p$ are contained in $L\cup LY$.   
When we present $p=e_1e_2\cdots e_n$, 
for $i=1,\ldots,n$, there exists $y_i\in Y$ 
such that $(e_i)_-\in Ly_{i-1}$ and $(e_i)_+\in Ly_i$.  
Then we consider a sequence 
$y_0^{-1}\phi (e_1)y_1y_1^{-1}\phi(e_2)y_2y_2^{-1}\cdots y_{n-1}y_{n-1}^{-1}\phi(e_n)y_n$
of elements of $X\sqcup Y\sqcup Y^{-1}\sqcup \mathcal H$. 
If $\phi(e_i)$ is an element of $X$, then we regard $y_{i-1}\phi(e_i)y_i^{-1}$ as an element of $W_1$.  
If $\phi(e_i)$ is an element of $\mathcal H$ and $y_{i-1}=y_{i}$, then 
we regard $y_{i-1}\phi(e_i)y_i^{-1}$ as an element of $\mathcal H_{L,Y}$. 
If $\phi(e_i)$ is an element of $\mathcal H$ and $y_{i-1}\neq y_{i}$, then 
we replace $y_{i-1}\phi(e_i)y_i^{-1}$ with a sequence of an element of $\mathcal H_{L,Y}$ and 
$l_{y_{i-1},y_i,H_i}$
where $\phi(e_i)\in H_i\setminus \{1\}$. 
Then we have a path $p'$ of $\overline{\Gamma}(L,\H_{L,Y},W_1\sqcup W_2)$ from $1$ to $l$
labeled by such a sequence. 
This shows that $W_1\sqcup W_2$ is a finite relative generating system
of $(L, \H_{L,Y})$.
Also we have $l(p')\le 3l(p)$.
On the other hand
we can easily confirm that the natural embedding 
$(L,d_{S\sqcup {\mathcal H}_{L,Y}})\to (G,d_{X\sqcup Y\sqcup Y^{-1}\sqcup S\sqcup {\mathcal H}})$ 
is large-scale Lipschitz.
Hence $L$ is undistorted relative to $\H$ in $G$ in view of Lemma \ref{choiceX}.  
\end{proof}

\subsection{Relatively quasiconvex subgroups of relatively hyperbolic subgroups}
In this section we prove Theorems \ref{cap}, \ref{qc-undist} and \ref{4}.
See Definition \ref{rqc} for definition of relatively quasiconvex subgroups.

Lemma \ref{choiceX} and Proposition \ref{k-similar} imply the following:
\begin{Cor}\label{choice'''}
Let $G$ be hyperbolic relative to $\H$.  
Let $L$ be pre-quasiconvex (resp. quasiconvex) relative to $\H$ in $G$. 
Then for any finite generating system $X$ of $(G,\H)$, 
$L$ is pre-quasiconvex (resp. quasiconvex) relative to $\H$ in $G$ with respect to $X$. 
\end{Cor}

The following is well-known for the case of countable groups 
(see for example \cite[Lemma 9.4]{Hru10}).
\begin{Lem}\label{teq}
Let $H$ and $K$ be subgroups of $G$. 
Suppose that $aH$ and $bK$ are arbitrary left cosets of subgroups of $G$.
For any finite subset $X$ of $G$, there exists a finite subset $Y$ of $G$ such that 
\[
a(H\cup HX)\cap b(K\cup KX)\subset (aHa^{-1}\cap bKb^{-1})\cup(aHa^{-1}\cap bKb^{-1})Y.
\]
\end{Lem}
\begin{proof}
We assume that $X$ has $1$ without loss of generality. 
For any element $z\in aHX\cap bKX$, 
there exists a pair $(x_1,x_2)\in X\times X$
such that $z=ahx_1=bkx_2$ for some pair $(h,k)\in H\times K$. 
We define a subset $M_{x_1,x_2}\subset aHX\cap bKX$
for any $(x_1,x_2)\in X\times X$ as the set of such elements. 
Then we have $aHX\cap bKX=\bigcup_{(x_1,x_2)\in X\times X} M_{x_1,x_2}$. 
For every $x_1,x_2\in X$ such that $M_{x_1,x_2}\neq \emptyset$, 
we fix an element $z_{x_1,x_2}=ah'x_1=bk'x_2\in M_{x_1,x_2}$. 
For any $z=ahx_1=bkx_2\in M_{x_1,x_2}$, we have 
$zz_{x_1,x_2}^{-1}=ahh'^{-1}a^{-1}=bkk'^{-1}b^{-1}\in aHa^{-1}\cap bKb^{-1}$. 
When we define $Y:=\{z_{x_1,x_2}\ | \ (x_1,x_2)\in X\times X, \ M_{x_1,x_2}\neq \emptyset\}$, 
we have $aHX\cap bKX\subset (aHa^{-1}\cap bKb^{-1})Y$.
\end{proof}

\begin{proof}[Proof of Theorem \ref{cap}]
Let $X$ be a finite relative generating system of $(G,\H)$. 
We take an arbitrary element $l \in  (K\cap L) \setminus \{1\}$ 
and denote by $p$ a geodesic from $1$ to $l$ in $\overline{\Gamma}(G, \H, X)$. 
Since both $K$ and $L$ are pre-quasiconvex relative to $\H$ in $G$, 
there exists a finite subset $Y$ of $G$ 
such that any vertex $v$ of $p$ is included in $(K\cup KY) \cap (L\cup LY)$. 
By Lemma \ref{teq}, 
there exists a finite subset $Z$ of $G$ such that 
$(K\cup KY) \cap (L\cup LY) \subset (K \cap L)\cup (K \cap L)Z$. 
We thus obtain $v \in (K\cap L)\cup(K\cap L)Z$.
Hence $K\cap L$ is pre-quasiconvex relative to $\H$ in $G$.

By Lemma \ref{cap'}, the group $K\cap L$ satisfies Condition $(b)$
with respect to $\H$ in $G$.
\end{proof}

\begin{proof}[Proof of Theorem \ref{qc-undist}]
First we prove the assertion (i).

It follows from Proposition \ref{rundqc}
that $L$ is undistorted relative to $\H$
if $L$ is quasiconvex relative to $\H$.

It follows from Proposition \ref{rundqc} and Proposition \ref{k-similar}
that $L$ is quasiconvex relative to $\H$
if $L$ is undistorted relative to $\H$.

Next we prove the assertion (ii).

We take a finite relative generating system $X$ of $(G,\H)$. 

Suppose that $L$ is strongly quasiconvex relative to $\H$ in $G$. 
By the assertion (i), $L$ is undistorted relative to $\H$ in $G$. 
We take a finite subset $Y$ of $G$ and a finite generating system $S'$ of $(L,\H_{L,Y})$
such that the natural embedding from $(L,d_{S'\sqcup \mathcal H_{L,Y}})$
to $(G,d_{X\sqcup \mathcal H})$ is quasi-isometric. 
We assume that all of the elements of $Y$ belong to mutually different 
right cosets of $L$ in $G$.  
By definition, $\H_{L,Y}$ is a finite set and every element is a finite subgroup of $L$. 
Hence $L$ has a finite generating system $S=S'\sqcup \mathcal H_{L,Y}$ and 
thus the natural embedding from $(L, d_{S})$ to $(G,d_{X\sqcup \mathcal H})$ is quasi-isometric.

Suppose that $L$ is strongly undistorted relative to $\H$ in $G$. 
We take a finite generating system $S$ of $L$. 
By taking $Y=\emptyset\subset G$, we recognize that 
$L$ is undistorted relative to $\H$ in $G$. 
By the assertion (i), $L$ is quasiconvex relative to $\H$ in $G$. 
Assume that 
the subgroup $L\cap gHg^{-1}$ is infinite for some $g\in G$ and $H\in\H$. 
Then $(L,d_S)$ and $(L,d_{S\sqcup \mathcal H_{L,\{g\}}})$ are not quasi-isometric. 
This contradicts Lemma \ref{choice}. 
Assume that $\H_{L,\{g\}}$ is infinite for some $g\in G$. 
Since $(L, \H_{L,\{g\}})$ satisfies Condition $(a)$ by Lemma \ref{2''}, 
$(L,d_S)$ and $(L,d_{S\sqcup \mathcal H_{L,\{g\}}})$ are not quasi-isometric. 
This contradicts Lemma \ref{choice}.
Hence $\H_{L,\{g\}}$ is a finite family consisting of finite subgroups of $L$
for any $g\in G$ 
and thus $L$ is strongly quasiconvex relative to $\mathbb{H}$ in $G$. 
\end{proof}

\begin{Ex}\label{Q'}
Let $G$ be hyperbolic relative to $\H$. 
We take a finite relative presentation $(X,{\mathcal R})$. 
Then the subgroup $Q$ in Example \ref{Q} 
is quasiconvex relative to $\H$ in $G$ 
by Theorem \ref{qc-undist}.
\end{Ex}

We have the following by Theorem \ref{qc-undist}
and Lemma \ref{conjundist}.
\begin{Cor}\label{conjqc}
Let $G$ be hyperbolic relative to $\H$.  
Then $L$ is (strongly) quasiconvex relative to $\H$ in $G$
if and only if $gLg^{-1}$ is (strongly) quasiconvex relative to $\H$ in $G$. 
\end{Cor}
\noindent
The case where $G$ is countable and $\H$ is finite was known (see \cite[Corollary 3.5]{MP-W11a}). 

\begin{Thm}\label{4'}
Let $G$ be hyperbolic relative to $\H$
and $X$ be a finite relative generating system of $(G,\H)$. 
Let $L$ be quasiconvex relative to $\H$ in $G$. 
Then for any finite subset $Y=\{y_1,\ldots,y_n\}\subset G$ 
such that all elements of $Y$ belong to mutually different right cosets of $L$ in $G$
and $L$ is finitely generated relative to $\H_{L,Y}$, 
$L$ is hyperbolic relative to $\H_{L,Y}^r$. 

Moreover for any finite subset 
$Y'=\{y_1,\ldots,y_n,y_{n+1}\ldots,y_{n'}\}\subset G$ and  
all elements of $Y'$ belong to mutually different right cosets of $L$ in $G$, 
$\H_{L,Y'}^r\setminus \H_{L,Y}^r$ is a finite family of finite subgroups of $L$.
\end{Thm}
\begin{proof}
Let $S$ be a finite relative generating system of $(L, \H_{L,Y})$.
By Theorem \ref{qc-undist} and Lemma \ref{red'}, $(L,\H_{L,Y}^r)$ satisfies Condition $(a)$
and $\widehat{\Gamma}(L,\H_{L,Y}^r,S)$ is hyperbolic. 
We prove that $\widehat{\Gamma}(L,\H_{L,Y}^r,S)$ is fine. 

First we define a correspondence $\iota$ 
from every vertex of $\widehat{\Gamma}(L,\H_{L,Y}^r,S)$ 
to a vertex of $\widehat{\Gamma}(G,\H,X\sqcup Y\sqcup Y^{-1}\sqcup S)$ as follows: 
\begin{align*}
\iota:L\ni l&\mapsto l\in G, \\
\iota:v(l(L\cap y_jHy_j^{-1}))&\mapsto v(ly_jH), 
\end{align*}
where $v(l(L\cap y_jHy_j^{-1}))$ is a cone-vertex of $\widehat{\Gamma}(L,\H_{L,Y}^r,S)$
and $v(ly_jH)$ is a cone-vertex of $\widehat{\Gamma}(G,\H,X\sqcup Y\sqcup Y^{-1}\sqcup S)$. 
The latter correspondence is well-defined. Indeed 
for any $l_1,l_2\in L$, $y_j\in Y$ and $H\in \H$ 
such that $L\cap y_jHy_j^{-1}\in \H_{L,Y}^r$, 
$l_1(L\cap y_jHy_j^{-1})=l_2(L\cap y_jHy_j^{-1})$ 
if and only if $l_1y_jH=l_2y_jH$.
This correspondence $\iota$ is injective. 
Indeed 
for any $l_1,l_2\in L$, $y_i,y_j\in Y$ and $H\in \H$ 
such that $L\cap y_iHy_i^{-1}, L\cap y_jHy_j^{-1}\in \H_{L,Y}^r$
and $i \le j$, 
if $l_1y_iH=l_2y_jH$, then we have $L\cap y_iHy_j^{-1}\neq \emptyset$. 
If $i<j$, then $L\cap y_jHy_j^{-1}\not\in \H_{L,Y}^r$ and thus $i=j$. 
This implies that $l_1(L\cap y_iHy_i^{-1})=l_2(L\cap y_jHy_j^{-1})$.

Next we extend $\iota$ to a correspondence from 
edges of $\widehat{\Gamma}(L,\H_{L,Y}^r,S)$ 
to paths of $\widehat{\Gamma}(G,\H,X\sqcup Y\sqcup Y^{-1}\sqcup S)$ as follows: 
\[
\iota:L\times S\ni (l,s)\mapsto (l,s)\in G\times (X\sqcup Y\sqcup Y^{-1}\sqcup S), 
\]
\begin{align*}
\iota:[l,v(l(L\cap y_jHy_j^{-1}))]&\mapsto (l,y_j)[ly_j, v(ly_jH)]\text{ for }y_j\in Y\setminus \{1\}, \\
\iota:[v(l(L\cap y_jHy_j^{-1})),l]&\mapsto [v(ly_jH),ly_j](ly_j,y_j^{-1})\text{ for }y_j\in Y\setminus \{1\}, 
\end{align*}
\begin{align*}
\iota:[l,v(l(L\cap 1H1^{-1}))]&\mapsto [l, v(lH)]\text{ for }y_j\in Y\cap \{1\}, \\
\iota:[v(l(L\cap 1H1^{-1})),l]&\mapsto [v(lH),l]\text{ for }y_j\in Y\cap \{1\}. 
\end{align*}

The above correspondence $\iota$ gives a faithful correspondence 
from each circuit of $\widehat{\Gamma}(L,\H_{L,Y}^r,S)$
to a cycle of $\widehat{\Gamma}(G,\H,X\sqcup Y\sqcup Y^{-1}\sqcup S)$.
Removing cycles of the form $(ly_j,y_j^{-1})(l,y_j)$ 
from the resulting cycle, we have a circuit by choice of $Y$. 
Such a circuit of $\widehat{\Gamma}(G,\H,X\sqcup Y\sqcup Y^{-1}\sqcup S)$ 
has at most twice length of the original circuit
of $\widehat{\Gamma}(L,\H_{L,Y}^r,S)$. 
Also this correspondence is faithful by choice of $Y$.
Each cycle containing an edge 
\[
(l,s),[l,v(l(L\cap y_jHy_j^{-1}))],
[v(l(L\cap y_jHy_j^{-1})),l],[l,v(l(L\cap 1H1^{-1}))],[v(l(L\cap 1H1^{-1})),l]
\]
gives a circuit containing an edge 
\[
(l,s),[ly_j, v(ly_jH)],
[v(ly_jH),ly],[l, v(lH)],[v(lH),l], 
\]  
respectively.
Since $\widehat{\Gamma}(G,\H,X\sqcup Y\sqcup Y^{-1}\sqcup S)$ is fine, 
$\widehat{\Gamma}(L,\H_{L,Y}^r,S)$ is also fine.

It follows from Lemma \ref{red''} that $\H_{L,Y'}^r\setminus \H_{L,Y}^r$ is a finite family.
Assume that we have an infinite subgroup $K\in \H_{L,Y'}^r\setminus \H_{L,Y}^r$. 
Then \cite[Theorem 1.5]{Osi06b} implies that 
$K$ is strongly undistorted relative to $\H_{L,Y}^r$ in $L$. 
In particular $K$ has infinite diameter in $(L, d_{S\sqcup \mathcal H_{L,Y}^r})$. 
On the other hand $K$ has finite diameter in $(L, d_{S\sqcup \mathcal H_{L,Y'}^r})$.
These contradict the fact that the identity map on $L$ is
quasi-isometric from $(L,d_{S\sqcup \mathcal H^r_{L,Y}})$ to
$(L,d_{S\sqcup \mathcal H_{L,Y'}})$ by Lemma \ref{red'} and Lemma \ref{choice}.
\end{proof}

\begin{Ex}\label{Q''}
Let $G$ be hyperbolic relative to $\H$. 
We take a finite relative presentation $(X,{\mathcal R})$. 
Then the subgroup $Q$ in Example \ref{Q} 
is hyperbolic relative to $\H_{Q,\emptyset}:=\{Q\cap H\neq \{1\} \ | \ H\in \H\}$
(\cite[Corollary 2.47]{Osi06a}).
Theorem \ref{4'} and Example \ref{Q'} give another proof for the fact. 
\end{Ex}

The following gives simple examples of relatively quasiconvex subgroups.  
\begin{Prop}\label{rel0hyp}
Let $G$ be $\ast_{H\in \H}H$. 
Suppose that $L$ is finitely generated relative to $\H$ in $G$, 
then $L$ is quasiconvex relative to $\H$ in $G$.
\end{Prop}
\begin{proof}
We consider $\widehat{\Gamma}(G,\H,\emptyset)$, which is a tree. 
Also we suppose that 
$(L,\H_{L,Y}^r)$ has a finite relative generating system $S$ 
where $Y$ is a finite subset of $G$ such that 
all elements of $Y$ belong to mutually different right cosets of $L$ in $G$. 
Then we consider a correspondence $\iota$ 
from every vertex of $\widehat{\Gamma}(L,\H_{L,Y}^r,S)$ 
to a vertex of $\widehat{\Gamma}(G,\H,\emptyset)$ as follows: 
\begin{align*}
\iota:L\ni l&\mapsto l\in G, \\
\iota:v(l(L\cap y_jHy_j^{-1}))&\mapsto v(ly_jH), 
\end{align*}
where $v(l(L\cap y_jHy_j^{-1}))$ is a cone-vertex of 
$\widehat{\Gamma}(L,\H_{L,Y}^r,S)$
and $v(ly_jH)$ is a cone-vertex of $\widehat{\Gamma}(G,\H,\emptyset)$. 
The correspondence is well-defined and injective (see Proof of Theorem \ref{4'}).

Since we have the geodesic $p_s$ of $\widehat{\Gamma}(G,\H,\emptyset)$ from $1$ to $s$ for each $s\in S$ 
and the geodesic $p_y$ of $\widehat{\Gamma}(G,\H,\emptyset)$ from $1$ to $y$ for each $y\in Y$, 
we extend $\iota$ to a correspondence from 
edges of $\widehat{\Gamma}(L,\H_{L,Y}^r,S)$ 
to paths of $\widehat{\Gamma}(G,\H,\emptyset)$ as follows: 
\[
\iota:L\times S\ni (l,s)\mapsto lp_s, 
\]
\begin{align*}
\iota:[l,v(l(L\cap y_jHy_j^{-1}))]&\mapsto lp_{y_j}[ly_j, v(ly_jH)], \\
\iota:[v(l(L\cap y_jHy_j^{-1})),l]&\mapsto [v(ly_jH),ly_j]ly_jp_{y_j}^{-1}.
\end{align*}
The above correspondence $\iota$ gives a faithful correspondence
from each arc of $\widehat{\Gamma}(L,\H_{L,Y}^r,S)$ between two elements of $L$
to a path of $\widehat{\Gamma}(G,\H,\emptyset)$.
Removing subpaths of the form $[ly_j, v(ly_jH)][v(ly_jH),ly_j]$ 
from the resulting path, we have an arc by choice of $Y$. 
Such an arc has at most twice length of the original arc. 
Note that the image of $\iota$, that is, 
\[
\bigcup_{l\in L} l\{(\bigcup_{s\in S}p_s )
\cup (\bigcup_{L\cap y_jHy_j^{-1}\in \H_{L,Y}^r} p_{y_j}[y_j,v(y_jH)])
\cup (\bigcup_{L\cap y_jHy_j^{-1}\in \H_{L,Y}^r} [v(y_jH),y_j]y_jp_{y_j}^{-1})\}
\]
is a tree. 
Hence the resulting arc is a geodesic. 
Thus the identity on $L$ is a quasi-isometry from $(L, d_{S\sqcup\mathcal H_{L,Y}^r})$ to 
$L$ with the induced metric by the graph $\widehat{\Gamma}(G,\H,\emptyset)$.
By Theorem \ref{qc-undist}, $L$ is quasiconvex relative to $\H$ in $G$.
\end{proof}

\subsection{The case where the family of subgroups is finite}\label{finite}
We deal with the case where $\H$ is finite.

First we give a corollary of Theorem \ref{4'}.
\begin{Cor}
In the setting in Theorem \ref{4'}, suppose that $\H$ is finite. 
Then $L$ is hyperbolic relative to a finite family of infinite subgroups of $L$
\[
\H_{L,Y}^{r,\infty}:=\{K\in \H_{L,Y}^r \ |\ K\text{ is infinite} \}.
\] 
Moreover 
$\H_{L,Y}^{r,\infty}$ represents conjugacy classes of 
\[
\{
K\subset L \ |\ K=L\cap gH'g^{-1}\text{ for some }H'\in \H, g\in G, 
K\text{ is infinite} 
\}.
\]
\end{Cor}
\begin{proof}
Since $\H_{L,Y}^r\setminus\H_{L,Y}^{r,\infty}$ is a finite family of finite subgroups, 
$L$ is hyperbolic relative to $\H_{L,Y}^{r,\infty}$ by Theorem \ref{4'}. 
Assume that there exist $y'\in G$ and $H'\in \H$
such that $L\cap y'H'y'^{-1}$ is infinite and is not conjugate to 
any element of $\H_{L,Y}^{r,\infty}$ in $L$.
Then all elements of $Y':=Y\sqcup \{y'\}$ belong to mutually different 
right cosets of $L$ in $G$ and thus 
$\H_{L,Y'}^r\setminus \H_{L,Y}^r$ is a finite family of finite subgroups
by Theorem \ref{4'}.
It contradicts choice of $y'$.
\end{proof}

The following is \cite[Definition 1.3]{MP-W11a}:
\begin{Def}\label{MPW}
Let a group $G$ be hyperbolic relative to a finite family $\H$ of infinite subgroups.  
A subgroup $L$ is {\it quasiconvex relative to} $\H$ in $G$ 
if for some $(G,\H)$-graph $\Gamma$, there exists a nonempty connected and 
quasi-isometrically embedded subgraph $\Lambda$ of $\Gamma$
that is $L$-invariant and has finitely many $L$-orbits of edges.
\end{Def}

\begin{Prop}
Let $G$ be a group with a finite family $\H$ of infinite subgroups.
Then relative quasiconvexity in the sense of Definition \ref{rqc} is 
equivalent to relative quasiconvexity in the sense of Definition \ref{MPW}.
\end{Prop}
\noindent
Since $\H$ is finite, 
Condition $(b)$ is automatically satisfied. 
This was known for the case where $G$ is countable
(see \cite[Theorem 1.7]{MP-W11a}).

\begin{proof}
Suppose that $L$ is quasiconvex relative to $\H$ in $G$ in the sense of Definition \ref{rqc}.
As in Proof of Theorem \ref{4'}, we consider 
the image of $\widehat{\Gamma}(L,\H_{L,Y}^r,S)$ 
in $\widehat{\Gamma}(G,\H,X\sqcup Y\sqcup Y^{-1}\sqcup S)$ by the correspondence $\iota$. 
It is a nonempty connected and 
quasi-isometrically embedded subgraph of $\widehat{\Gamma}(G,\H,X\sqcup Y\sqcup Y^{-1}\sqcup S)$
that is $L$-invariant and has finitely many $L$-orbits of edges.

Suppose that $L$ is quasiconvex relative to $\H$ in $G$ in the sense of Definition \ref{MPW}.
Then we have a nonempty connected and 
quasi-isometrically embedded subgraph $\Lambda$ of 
a coned-off Cayley graph 
that is $L$-invariant and has finitely many $L$-orbits of edges
by \cite[Proposition 4.3 and Theorem 2.14]{MP-W11a}.
We can apply the proof of \cite[Proposition 4.15]{MP-W11a} to our case 
by using Proposition \ref{k-similar}
instead of \cite[Proposition 4.4]{MP-W11a}.  
\end{proof}

%%%%%%%%%%%%%%%%%%%%%%%%%%%

\end{document}